\documentclass[12pt]{article}
\usepackage[utf8]{inputenc}
\usepackage[english]{babel}
\usepackage{amsmath, amsthm}
\usepackage{amsfonts}
\usepackage{amssymb}
\usepackage{authblk}
\usepackage{cases}
\usepackage{xcolor}
\usepackage{graphicx}
\numberwithin{equation}{section}

\newcommand{\be}{\begin{eqnarray}}
\newcommand{\ee}{\end{eqnarray}}
\newcommand{\ce}{\begin{eqnarray*}}
\newcommand{\de}{\end{eqnarray*}}
\newtheorem{theorem}{Theorem}[section]
\newtheorem{lemma}[theorem]{Lemma}
\newtheorem{remark}[theorem]{Remark}

\newtheorem{example}[theorem]{Example}

\newtheorem{assumption}{Assumption}
\allowdisplaybreaks

\begin{document}

\title{ {\bf SDEs with singular drifts and multiplicative noise on general space-time domains}}
\date{}
\author[a,c]{Chengcheng Ling}
\author[c]{Michael R\"ockner}
\author[b,c]{Xiangchan Zhu\thanks{Research of C. Ling and M. R\"ockner is  supported  by the DFG through the IRTG 2235 Bielefeld-Seoul “Searching for the regular in the
irregular: Analysis of singular and random systems.”\\
\indent Research of X. Zhu is supported by NSFC (11771037).\\
\indent Email: cling@math.uni-bielefeld.de (C. Ling),  roeckner@math.uni-bielefeld.de (M. R\"ockner), zhuxiangchan@126.com (X. Zhu)}}
\affil[a]{School of science, Beijing Jiaotong University, Beijing 100044, China}
\affil[b]{Academy of Mathematics and Systems Science, Chinese Academy of Sciences, Beijing 100190, China}
\affil[c]{Department of Mathematics, University of Bielefeld, D-33615, Germany}
\maketitle
\section*{Abstract}In this paper, we prove the existence and uniqueness of  maximally defined strong solutions to SDEs driven by multiplicative noise on general space-time domains $Q\subset\mathbb{R}_+\times\mathbb{R}^d$, which have continuous paths on the one-point compactification $Q\cup\partial$ of $Q$ where $\partial \notin Q$ and $Q\cup\partial$ is equipped with the Alexandrov topology. If the SDE is of gradient type (see \eqref{eq2.4} below) we prove that under suitable Lyapunov type conditions the life time of the solution is infinite and its  distribution has  sub-Gaussian tails. This generalizes earlier work \cite{KR} by Krylov and one of the authors to the case where the noise is multiplicative.
\section*{Key words} Krylov's estimate, stochastic differential equation, well-posedness, non-explosion of the solution, maximally defined local solution to SDE, singular drift, multiplicative noise

\section{Introduction}

Consider the following stochastic differential equation (abbreviated as SDE):
\begin{equation}\label{eq1}
X_t=x+\int_0^tb(s+r,X_r)dr+\int_0^t\sigma(s+r,X_r)dW_r,  \quad t\geq 0,
\end{equation}
in an open set $Q\subset[0,\infty)\times\mathbb{R}^{d}$ with measurable coefficients $b=(b_i)_{1\leq i\leq d}: Q\rightarrow \mathbb{R}^d$ and $\sigma=(\sigma_{ij})_{1\leq i,j\leq d}: Q\rightarrow L(\mathbb{R}^d)$ $(:=d\times d$ real valued matrices$)$. Here $(s,x)\in Q$ is the initial  point, and $(W_t)_{t\geq0}$ is a $d-$dimensional $(\mathcal{F}_t)$-Wiener process defined on a complete filtered probability space $(\Omega,\mathcal{F},(\mathcal{F}_t)_{t\geq0},P)$.
 Define
$$\xi := \inf\left\{t\geq 0: (t+s,X_t)\notin Q\right\}.$$
 $\xi$ is called the explosion time (or life time) of the process $(t+s,X_t)_{t\geq 0}$ in the domain $Q$.

  There are many known results on studying existence and uniqueness of strong solutions to the SDE $\eqref{eq1}$.  In the seminal paper \cite{A. J.},  Veretennikov  proved that for $Q=\mathbb{R}_+\times\mathbb{R}^d$, if the coefficient $\sigma$ is Lipschitz continuous in the space variable $x$ uniformly with respect to the time variable $t$, $\sigma\sigma^*$ is uniformly elliptic, and $b$ is bounded and measurable, then the SDE $\eqref{eq1}$ admits a unique global strong solution (i.e. $\xi=\infty$ $a.s.$ where $\xi$ is defined as above).  In  \cite{KR} under the assumptions that $\sigma=\mathbb{I}_{d\times d}$ (i.e. additive noise,  $\mathbb{I}_{d\times d}$ denotes the unit matrix) and $|bI_{Q^n}|\in L^{q(n)}(\mathbb{R};L^{p(n)}(\mathbb{R}^d))$ for $p(n),q(n)\in(2,\infty)$ and $d/p(n)+2/q(n)<1$, where $Q^n$ are open bounded subsets of $Q$ with $\overline{Q^n}\subset Q^{n+1}$ and $Q=\cup_{n}Q^n$,  Krylov and R\"ockner proved the existence of a unique maximal local strong solution to $\eqref{eq1}$ when $Q$ is a subset of $\mathbb{R}^{d+1}$, in the sense that there exists a unique strong solution $(s+t,X_t)_{t\geq0}$ solving \eqref{eq1} on $[0,\xi)$ such that $[0,\infty)\ni t\rightarrow (s+t,X_t)\in Q' :=Q\cup \partial$ (=Alexandrov compactification of $Q$) is continuous and this process is defined to be in $\partial $ if $t\geq \xi$. To this end they applied the Girsanov transformation to get existence of a weak solution firstly and  then proved pathwise uniqueness of \eqref{eq1} by Zvonkin's transformation invented in \cite{Zvonkin}. Then, the well-known Yamada-Watanabe theorem \cite{YW} yields existence and uniqueness of a maximal local strong solution.  Fedrizzi and Flandoli \cite{FF} introduced a new method to prove existence and uniqueness of a global strong solution to the SDE \eqref{eq1} by using regularizing properties of solutions to the Kolmogrov equation corresponding to \eqref{eq1}, assuming that $\sigma=\mathbb{I}_{d\times d}$, $|b|\in L_{loc}^q(\mathbb{R}_+,L^p(\mathbb{R}^d))$ with $p,q\in(2,\infty)$ and $d/p+2/q<1$. This method was extended by  Von der L\"uhe to the multiplicative noise case in her work \cite{K}.  Zhang in \cite{Zhang2011} proved existence and uniqueness of a strong solution to the SDE \eqref{eq1} on $Q=\mathbb{R}_+\times\mathbb{R}^d$ for $t<\tau$ $a.s.$, where $\tau$ is some stopping time, under the assumptions that $\sigma$ is  bounded and uniformly continuous in $x$ locally uniformly with respect to $t$, $\sigma\sigma^*$ is uniformly elliptic,  and $|b|,|\nabla\sigma|\in L_{loc}^{q(n)}(\mathbb{R}_+;L^{p(n)}(B_n))$ (where $\nabla \sigma$ denotes the weak gradient of $\sigma$ with respect to $x$) with $p(n),q(n)$ satisfying $p(n),q(n)\in(2,\infty)$ and $d/p(n)+2/q(n)<1$, where $B_n$ is the ball in $\mathbb{R}^d$ with radius $n\in\mathbb{N}$ centered at zero. Zvonkin's transformation plays a crucial role in Zhang's proof (see also \cite{A. Yu}, \cite{Zhang 2005},  \cite{Zhang Xie} for further interesting results on this topic, which however do not cover our results in this paper).  The  above results include the case where the coefficients of the SDE \eqref{eq1} are time dependent.   For the time independent case, Wang \cite{Wang} and Trutnau \cite{Trutnau} used generalized Dirichlet forms to get existence and uniqueness and also non-explosion results for the SDE \eqref{eq1} on $Q=\mathbb{R}^d$.\\
    \indent As mentioned in \cite{KR}, there are several interesting situations arising from applications, say diffusions in random media and particle systems, where the domain $Q$ of \eqref{eq1} is not the full space $\mathbb{R}\times\mathbb{R}^d$ but a subdomain (e.g. $ Q=\mathbb{R}\times(\mathbb{R}^{d}\backslash\gamma^\rho)$, where $\gamma^\rho=\{x\in\mathbb{R}^d|dist(x,\gamma)\leq \rho\}$, $\rho>0$, and $\gamma$ is a locally finite subset of $\mathbb{R}^d$), where none of the above results mentioned can be applied to get global solutions, except for the one in \cite{KR}. Moreover, Krylov and R\"ockner in \cite{KR} not only proved the existence and uniqueness of a maximal local strong solution of the equation on $Q$, but also they obtained that if $b=-\nabla\phi$, i.e., $b$ is minus the gradient in space of a nonnegative function $\phi$ and if there exist a constant $K\in[0,\infty)$ and an integrable function $h$ on $Q$ defined as above such that the following Lyapunov conditions hold in the distributional sense
\begin{equation}\label{eq2}
2D_t\phi\leq K\phi,\quad 2D_t\phi +\Delta \phi\leq he^{\epsilon \phi}, \quad \epsilon\in [0,2),
\end{equation}
the strong solution does not blow up, which means $\xi=\infty$ $a.s.$. Here $D_t\phi$ denotes the derivative of $\phi$ with respect to $t$. This result can be applied to diffusions in random environment and also finite interacting particle systems to show that if the above Lyapunov conditions hold, the process does not exit from $Q$ or go to infinity in finite time. However, \cite{KR} is restricted to the case where the equation \eqref{eq1} is driven by additive noise, that is, the diffusion term is a Brownian motion. \\
 \indent Our aim in this paper is to extend these results on existence and uniqueness of maximally defined local solutions and also the non-explosion results  in \cite{KR} to the multiplicative noise case on general space-time  domains $Q$. In order to prove the maximal local well-posedness result, we use a localization technique and the well-posedeness result in \cite{Zhang2011}. We want to point out that  as Krylov and R\"ockner did in \cite{KR}, we also prove the continuity of the paths of the solution not only in the domain $Q$ but also on $Q'=Q\cup\partial$, which  essentially follows from Lemmas \ref{ossilaton1} and \ref{ossilation2} below. As far as the non-explosion result is concerned, we have to take into account that having  non-constant $\sigma$ instead of $\mathbb{I}_{d\times d}$ in front of the Brownian motion in \eqref{eq1} means that we have to consider a different geometry on $\mathbb{R}^d$, and that this effects the Lyapunov function type conditions which are to replace \eqref{eq2} and also the form of the equation. In Remark \ref{gk} by comparing the underlying Kolmogrov operators, we explain why the SDE \eqref{eq2.4} should be considered and why \eqref{eq5} states the right Lyapunov type conditions which are analog to the ones in \eqref{eq2}. This leads to some substantial changes in the proof of our non-explosion result in comparison with the one in \cite{KR}.    In addition, we give some examples to show our  well-posedness and non-explosion results in Sections 6.1 and 6.2. We also give two  applications to diffusions in random media and particle systems. Both are  generalizations of the examples in \cite [Section 9]{KR} to the case of  multiplicative noise.

The organization of this paper is as follows: We state our notions and  main results in Section 2 . In Section 3 we prove that there exists a pathwise unique maximal strong solution $(s+t,X_t)_{t\geq0}$ solving the SDE \eqref{eq1} on $[0,\xi)$, and that the paths of  $(s+t,X_t)_{t\geq0}$ are continuous in $Q'=Q\cup\partial$.   Section 4 is devoted to the preparation of the proof of our non-explosion result, which is subsequently proved in Section 5. We discuss  examples and applications of our results in Section 6 . The Appendix contains technical lemmas used in the proofs of our main results.

\subsection*{Acknowledgement}
The authors are grateful to Prof. Fengyu Wang and Dr. Guohuan Zhao for helpful discussions.

\section{Main results}

\indent Let $Q$ be an open subset of $\mathbb{R}_+\times\mathbb{R}^{d}$ and $Q^n$, $n\geq1$, be bounded open subsets of $Q$ such that $\overline{Q^n}\subset Q^{n+1}$ and $\cup_{n}Q^n=Q$. We add an object $\partial\notin Q$ to $Q$ and define the neighborhoods of $\partial$ as the complements in $Q$ of closed bounded subsets. Then $Q'=Q\cup\partial$ becomes a compact topological space, which is just the Alexandrov compactification of $Q$.  For $p$, $q\in[1,\infty)$ and $0\leq S<T<\infty$, let $\mathbb{L}_p^q(S,T)$ denote the space of all real Borel measurable functions on $[S,T]\times\mathbb{R}^d$ with the norm
$$\Vert f\Vert_{\mathbb{L}_p^q(S,T)}:=\bigg(\int_S^T\big(\int_{\mathbb{R}^d}|f(t,x)|^pdx\big)^{q/p}dt\bigg)^{1/q}<+\infty.$$
 For simplicity, we write
$$\mathbb{L}_p^q=\mathbb{L}_p^q(0,\infty),\quad\mathbb{L}_p^q(T)=\mathbb{L}_p^q(0,T),\quad\mathbb{L}_p^{q,loc}=L_q^{loc}(\mathbb{R}_+,L_p(\mathbb{R}^d)).$$
Let $\mathcal{C}([0,\infty),\mathbb{R}^d)$ denote the space of all continuous $\mathbb{R}^d$-valued functions defined on $[0,\infty)$, by $\mathcal{C}([0,\infty),Q')$ we denote all continuous $Q'-$valued paths, $\mathcal{C}_b^n(\mathbb{R}^d)$ denotes the set of all bounded $n$ times continuously differentiable functions on $\mathbb{R}^d$ with bounded derivatives of all orders. Set $(a_{ij})_{1\leq i,j\leq d}:=\sigma\sigma^*$, where $\sigma^*$ denotes the transpose of $\sigma$. For $f\in L_{loc}^1(\mathbb{R}^d)$ we define $\partial_jf(x):=\frac{\partial f}{\partial x_j}(x)$ and  $\nabla f:=(\partial_i f)_{1\leq i\leq d}$ denotes the gradient of $f$. Here the derivatives are meant in the sense of distributions. For a real valued function $g\in\mathcal{C}^1([0,\infty))$, $D_tg$ denotes the derivative of $g$ with respect to $t$.  $L(\mathbb{R}^d)$ denotes all $d\times d$ real valued matrices.
%The appearance of $\frac{1}{2}\sum_{j=1}^d\partial_ja_{ij}(t,X_t)$ always means the vector $(\frac{1}{2}\sum_{j=1}^d\partial_ja_{ij}(t,X_t))_{1\leq i\leq d}$ for simplicity.
\\

\indent We first state the result about maximally local well posedness of the SDE \eqref{eq1} on a domain $Q\subset\mathbb{R}_+\times\mathbb{R}^{d}$.
\begin{theorem}\label{mainadd1}
Let $(W_t)_{t\geq0}$ be an $d-$dimensional Wiener process defined on a complete probability space $(\Omega,\mathcal{F}, P)$, let $(\mathcal{F}_t)_{t\geq0}=(\mathcal{F}_t^W)_{t\geq0}$ be the normal filtration generated by $(W_t)_{t\geq0}$.  Assume that for any $n\in\mathbb{N}$ and some $p_n,$ $q_n\in(2,\infty)$, satisfying $d/p_n+2/q_n<1,$\\
(i) $|bI_{Q^n}|$, $|I_{Q^n}\nabla\sigma |\in \mathbb{L}_{p_n}^{q_n}$,\\
(ii) For all $1\leq i,j\leq d$, $Q\ni(t,x)\rightarrow \sigma_{ij}(t,x)\in\mathbb{R}$ is continuous in $x$ uniformly with respect to $t$ on $Q^n$, and there exists a positive constant $\delta_n$ such that for all $(t,x)\in Q^n$,
$$|\sigma^*(t,x)\lambda|^2\geq \delta_n|\lambda|^2, \quad \forall \lambda\in\mathbb{R}^d.$$
Then for any $(s,x)\in Q$, there exists  an $(\mathcal{F}_t)$-stopping time $\xi:=\inf\left\{t\geq0:z_t\notin Q\right\}$ and an $(\mathcal{F}_t)$-adapted, pathwise unique and $Q'$-valued  process $(z_t)_{t\geq0}:=(s+t,X_t)_{t\geq0}$ which is continuous in $Q'$  such that
\begin{equation}\label{eq1add}
{X_t=x+\int_0^tb(s+r,X_r)dr+\int_0^t\sigma(s+r,X_r)dW_r, \quad \forall t\in [0,\xi), a.s. }
\end{equation}
and for any $t\geq0$, $z_t=\partial $ on the set $\left\{\omega:t\geq\xi(\omega)\right\}(a.s.)$.
\end{theorem}
\begin{remark}
In above theorem the condition $p,q\in(2,\infty)$ is automatically fulfilled when $d\geq 2$ since we also assume $d/p+2/q<1$. When $d=1$, we can refer to the result from Engelbert and Schmidt \cite{ES}  to obtain the existence and uniqueness of a strong solution to homogeneous SDE  on $\mathbb{R}^d$. They proved that if $\sigma(x)\neq0$ for all $x\in\mathbb{R}$ and $b/\sigma^2\in L_{loc}^1(\mathbb{R})$, and there exists a constant $C>0$ such that
$$|\sigma(x)-\sigma(y)|\leq C\sqrt{|x-y|},\quad x,y\in\mathbb{R},$$
$$|b(x)|+|\sigma(x)|\leq C(1+|x|),$$
then there exists a pathwise unique and $(\mathcal{F}_t)$-adapted process $(X_t)_{t\geq0}$ such that the SDE  $X_t=x+\int_0^tb(X_t)dt+\int_0^t\sigma(X_t)dW_t$ holds $a.s..$
\end{remark}
 Below we will give the non-explosion result for the solution to an SDE which is in a special form of \eqref{eq1add} on a domain $Q\subset\mathbb{R}_+\times\mathbb{R}^{d}$ under the following assumptions.
\begin{assumption}\label{ass2.1}
\noindent\textbf{(i)} $\phi$ is a nonnegative continuous function  defined on $Q$.\\
\textbf{(ii)} For each $n$ there exist $p=p(n)$, $q=q(n)$ satisfying
\begin{equation}\label{eq3}
p,q\in(2,\infty),\quad \frac{d}{p}+\frac{2}{q}<1,
\end{equation}
such that $|I_{Q^n}\nabla\phi |$, $|I_{Q^n}\nabla\sigma |\in \mathbb{L}_{p}^{q}$.\\
\textbf{(iii)} For each $1\leq i,j\leq d$, $Q\ni(t,x)\rightarrow \sigma_{ij}(t,x)\in\mathbb{R}$ is uniformly continuous in $x$ locally uniformly with respect to $t$, and there exists a positive constant $K$ such that for all $(t,x)\in Q$,
$$\frac{1}{K}|\lambda|^2\leq|\sigma^*(t,x)\lambda|^2\leq K|\lambda|^2,\quad \forall \lambda\in\mathbb{R}^d.$$
\textbf{(iv)} For some constants $K_1\in[0,\infty)$ and $\epsilon\in [0,2)$, in the sense of distributions on $Q$ we have
\begin{equation}\label{eq5}
2D_t\phi\leq K_1\phi,\quad 2D_t\phi+\sum_{i,j=1}^d\partial_j(a_{ij}\partial_{i}\phi)\leq he^{\epsilon\phi}.
\end{equation}
where $h$ is a continuous nonnegative function defined on $Q$ satisfying the following condition:\\
(H)\label{assH} For any $a>0$ and $T\in (0,\infty)$ there is an $r=r(T,a)\in(1,\infty)$ such that
$$H(T,a,r):=H_Q(T,a,r):=\int_ Qh^r(t,x)I_{(0,T)}(t)e^{-a|x|^2}dtdx<\infty.$$
\textbf{(v)}  For all $1\leq i,j\leq d$, for all $(t,x)$, $(s,y)\in Q$,
\begin{align}\label{eqK-lip}
|a_{ij}(t,x)-a_{ij}(s,y)|&\leq K(|x-y|\vee|t-s|^{1/2}),
\end{align}
and for all $n\in\mathbb{N}$,  and $(t,x)$, $(s,y)\in Q^n$, there exists $C_n\in[0,\infty)$ such that
$$|\partial_ja_{ij}(t,x)-\partial_ja_{ij}(s,y)|\leq C_n(|x-y|\vee|t-s|^{1/2}).$$
\textbf{(vi)}\label{addvi} The function $\phi$ blows up near the parabolic boundary of $Q$, that is for any $(s,x)\in Q$, $\tau\in(0,\infty)$, and any continuous bounded $\mathbb{R}^d-$valued function $x_t$ defined on $[0,\tau)$ and such that $(s+t,x_t)\in Q$ for all $t\in[0,\tau)$ and
\begin{align*}
\liminf_{t\uparrow\tau}dist((s+t,x_t),\partial Q)=0,
\end{align*}
we have
\begin{align*}
\limsup_{t\uparrow \tau}\phi(s+t,x_t)=\infty.
\end{align*}

\end{assumption}

\begin{remark}\label{rem2.2}
Observe that $H(T,a,r)<\infty$ if $h$ is just a constant. Furthermore, Assumption \ref{ass2.1} (iii) shows that $\sigma$ is bounded on $Q$, invertible for every $(t,x)\in Q$, and the inverse $\sigma^{-1}$ is also bounded on $ Q$.
\end{remark}
\begin{theorem}\label{th2.3}
Let Assumption \ref{ass2.1}  be satisfied. Let $(\Omega,\mathcal{F},(\mathcal{F}_t)_{t\geq0},P)$ and $(W_t)_{t\geq0}$  be as in Theorem \ref{mainadd1}.  Then for any $(s,x)\in Q$ there exists a continuous $\mathbb{R}^d$-valued and $(\mathcal{F}_t)$-adapted random process $(X_t)_{t\geq0}$ such that almost surely for all $t\geq0$, $(s+t,X_t)\in Q, \quad $
 \begin{equation}\label{eq2.4}
 X_t=x+\int_0^t(-\sigma\sigma^*\nabla\phi)(s+r,X_r)dr+\frac{1}{2}(\sum_{j=1}^d\int_0^t\partial_ja_{ij}(s+r,X_r)dr)_{1\leq i\leq d}+\int_0^t\sigma(s+r,X_r)dW_r.
 \end{equation}
 Furthermore, for each $T\in(0,\infty)$ and $m\geq 1$ there exists a constant $N$, depending only on $K$, $K_1$, $d$, $p(m+1)$, $q(m+1)$, $\epsilon$, $T$, $\Vert I_{Q^{m+1}}\nabla\phi \Vert_{\mathbb{L}_{p(m+1)}^{q(m+1)}}$, $dist(\partial Q^m,\partial Q^{m+1})$, $\sup_{Q^{m+1}}\left\{\phi+h \right\}$, and the function $H$, such that for $(s,x)\in Q^m$, $t\leq T$, we have
 $$E\sup_{t\leq T}\exp(\mu\phi(s+t,X_t)+\mu\nu|X_t|^2)\leq N,$$
  where
\begin{equation}\label{eq2.5}\mu=(\delta/2)e^{-TK_1/(2\delta)},\quad \delta=1/2-\epsilon/4,\quad \nu=\mu/(12KT).
  \end{equation}
\end{theorem}

\begin{remark}\label{gk}
Obviously, the Kolmogrov operator $\mathcal{L}$ corresponding to \eqref{eq2.4} is given by
\begin{align}\label{K}
\mathcal{L}=div(\sigma\sigma^*\nabla)+\langle\sigma^*\nabla,\sigma^*\nabla\rangle,
\end{align}
where $\langle\cdot,\cdot\rangle$ denotes the inner product in $\mathbb{R}^d$. Recalling that $div{\small\circ}\sigma$ is the adjoint of the 'geometric' gradient $\sigma^*\nabla$ (i.e. taking into account the geometry given to $\mathbb{R}^d$ through $\sigma$), we see that \eqref{eq2.4} is the geometrically correct analogue of the SDE
$$dX_t=-\nabla \phi(X_t)dt+dW_t,\quad t\geq0$$ studied in \cite{KR}. So, the Laplacian $\Delta$ in \cite{KR} is replaced by the Laplace-Beltrami operator $div(\sigma\sigma^*\nabla)(=\sum_{i,j=1}^d\partial_j(a_{ij}\partial_i))$ and the Euclidean gradient $\nabla$ in \cite{KR} is replaced by the 'geometric' gradient $\sigma^*\nabla$. Also condition \eqref{eq5} is then the exact analogue of condition \eqref{eq2} above, which was assumed in \cite{KR}.

\end{remark}

\section{Existence and uniqueness of a maximal local strong solution to the SDE \eqref{eq1} on an arbitrary domain in $\mathbb{R}_+\times\mathbb{R}^{d}$ }
Theorem \ref{mainadd1} says that there exists a unique maximally local strong solution to the SDE \eqref{eq1}. Before going to its proof we give some results as preparation.
\subsection{Preparation}
 Consider the SDE \eqref{eq1} in $[0,\infty)\times\mathbb{R}^d.$ First we recall two results from \cite{Zhang2011}.\\
\begin{lemma}\label{lemm3.1} (\cite[Theorem 1.1]{Zhang2011})
Assume that $p$, $q\in(2,\infty)$ satisfying $d/p+2/q<1$ and the following conditions hold.\\
(i) $|b|$, $|\nabla\sigma|\in\mathbb{L}_p^{q,loc} $.\\
(ii) For all $1\leq i,j\leq d$, $[0,\infty)\times\mathbb{R}^d\ni(t,x)\rightarrow \sigma_{ij}(t,x)\in\mathbb{R}$ is uniformly continuous in $x$ locally uniformly with respect to $t\in[0,\infty)$, and there exist positive constants $K$ and $\delta$ such that for all $(t,x)\in [0,\infty)\times{\mathbb{R}^d}$
\begin{align}\label{uniell}
\delta|\lambda|^2\leq|\sigma^*(t,x)\lambda|^2\leq K|\lambda|^2,\quad \forall \lambda\in\mathbb{R}^d.
\end{align}
Then for any $(\mathcal{F}_t)-$stopping time $\tau$ and $x\in\mathbb{R}^d$, there exists a unique $(\mathcal{F}_t)$-adapted continuous $\mathbb{R}^d$-velued process $(X_t)_{t\geq0}$ such that
\begin{align}\label{solutioninte}
P\left\{\omega:\int_0^T|b(r,X_r(\omega))|dr+\int_0^T|\sigma(r,X_r(\omega))|^2dr<\infty,\forall T\in[0,\tau(\omega))\right\}=1,
\end{align}
and
\begin{align}\label{eqadd3}
X_t=x+\int_0^tb(r,X_r)dr+\int_0^t\sigma(r,X_r)dW_r, \quad \forall t\in [0,\tau) \quad a.s,
\end{align}
which means that if there is another $(\mathcal{F}_t)$-adapted continuous stochastic process $(Y_t)_{t\geq0}$ also satisfying \eqref{solutioninte} and \eqref{eqadd3}, then
\begin{align*}
P\left\{\omega:X_t(\omega)=Y_t(\omega),\forall t\in[0,\tau(\omega))\right\}=1.
\end{align*}
Moreover, for almost all $\omega$ and all $t\geq0$, $ x\rightarrow X_t(\omega,x)$ is a homeomorphism on $\mathbb{R}^d$
and there exists a function $t\rightarrow C_t\in(0,\infty)$ such that $C_t\rightarrow \infty$ as $t\rightarrow \infty$ and for all $t>0$ and all bounded measurable function $\psi$, for $x$, $y\in\mathbb{R}^d$,
\begin{align*}
|E\psi(X_t(x))-E\psi(X_t(y))|\leq C_t\Vert\psi\Vert_{\infty}|x-y|.
\end{align*}
\end{lemma}

\indent Below we shall make essential use of Krylov's estimate. Therefore, we recall them here for reader's convenience.

\begin{lemma}\label{lemm3.2} (\cite[Theorem 2.1, Theorem 2.2]{Zhang2011})
Suppose $\sigma$ satisfies the conditions in  Lemma \ref{lemm3.1} and  let $(X_t)_{t\geq 0}$ be continuous and $(\mathcal{F}_t)$-adapted $\mathbb{R}^d$-valued process satisfying \eqref{solutioninte} and \eqref{eqadd3}. Fix an $(\mathcal{F}_t)$-stopping time  $\tau$ and let $T_0>0$.\\
(1) If $b$ is  Borel measurable and bounded, then for $p,$ $q\in (1,\infty)$ with
$$\frac{d}{p}+\frac{2}{q}<2,$$ then there exists a positive constant $N=N(K,d,p,q,T_0,\Vert b\Vert_{\infty})$ such that for all $f\in \mathbb{L}_p^q(T_0)$ and $0\leq S<T\leq T_0$,
\begin{equation}\label{eq3.2}
E\left(\int_{S\wedge\tau}^{T\wedge\tau}|f(s,X_s)|ds\bigg|{\mathcal{F}_S}\right)\leq N\Vert f\Vert_{\mathbb{L}_p^q(S,T)}.
\end{equation}
(2) If $b\in \mathbb{L}_p^q$ provided with
\begin{align}\label{b}
\frac{d}{p}+\frac{2}{q}<1,\quad p,q\in(1,\infty),
\end{align}
then there exists a positive constant $N=N(K,d,p,q,T_0,\Vert b\Vert_{\mathbb{L}_p^q(T_0)})$ such that for all $f\in \mathbb{L}_p^q(T_0)$ and $0\leq S<T\leq T_0$,
$$E\left(\int_{S\wedge\tau}^{T\wedge\tau}|f(s,X_s)|ds\bigg|{\mathcal{F}_S}\right)\leq N\Vert f\Vert_{\mathbb{L}_p^q(S,T)}.$$
\end{lemma}

We note that  actually condition $f\in\mathbb{L}_{p}^q(T_0)$ with $p,q\in(1,\infty)$ and $\frac{d}{p}+\frac{2}{q}<1$  in the above Lemma \ref{lemm3.2} can be improved to $f\in\mathbb{L}_{p'}^{q'}(T_0)$ with $p',q'\in(1,\infty)$ and $\frac{d}{p'}+\frac{2}{q'}<2$  without assuming that $b$ is bounded, which we shall prove in the following lemma. Let $K_0$ and $T_0$ be some positive constants and we give the following assumption.
\\
\begin{assumption}\label{ass3.2}(i)  For all $1\leq i,j\leq d$, $[0,\infty)\times\mathbb{R}^d\ni(t,x)\rightarrow \sigma_{ij}(t,x)\in\mathbb{R}$ is uniformly continuous in $x$ locally uniformly with respect to $t\in[0,\infty)$, and there exist positive constants $K$ and $\delta$ such that for all $(t,x)\in [0,\infty)\times{\mathbb{R}^d}$
\begin{align}\label{uniell1}
\delta|\lambda|^2\leq|\sigma^*(t,x)\lambda|^2\leq K|\lambda|^2,\quad \forall \lambda\in\mathbb{R}^d.
\end{align} And $|\nabla\sigma|\in\mathbb{L}_p^{q,loc} $ with $p$, $q\in(2,\infty)$ satisfying $d/p+2/q<1$.\\
(ii)  $b(t,x)$ is Borel measurable with $\Vert b\Vert_{\mathbb{L}_p^q}\leq K_0$  and $b(t,x)=0$ for $t>T_0$.
\end{assumption}

\begin{lemma}\label{lemm3.4}
Let Assumption \ref{ass3.2} hold. Let $(X_t)_{t\geq0}$  be a continuous $(\mathcal{F}_t)$-adapted process such that \eqref{solutioninte} and \eqref{eqadd3} are satisfied.
Then for any Borel function $f\in{\mathbb{L}_{p'}^{q'}(S,T)}$ with $p',q'\in(1,\infty)$ and $d/p'+2/q'<2$, and for $0\leq S<T\leq T_0$, we have
\begin{equation}\label{eq3.42}
E\int_{S}^{T} |f(t,X_t)|dt\leq N(d,p',q',K,\Vert b\Vert_{\mathbb{L}_p^q(T_0)})\Vert f\Vert_{\mathbb{L}_{p'}^{q'}(S,T)}.
\end{equation}
Furthermore, for any constant $\kappa\geq 0$ and $g\in \mathbb{L}_{p}^{q}(T_0),$
\begin{align}\label{expx}
E\exp(\kappa\int_0^{T_0}|g(t,X_t)|^2dt)<\infty.
\end{align}
\end{lemma}

\begin{proof}
   By Lemma \ref{lemm3.1} we obtain that there exists a unique  $(\mathcal{F}_t)$-adapted $\mathbb{R}^d$-valued process $(M_t)_{t\geq0}$ such that $M_t=x+\int_0^t\sigma(s,M_s)dW_s$, $t\geq0$.  For any $p_1$, $q_1\in(1,\infty)$ satisfying
$$\frac{d}{p_1}+\frac{2}{q_1}<2,$$  Lemma \ref{lemm3.2} implies that for $0<S<T\leq T_0$, and $f\in \mathbb{L}_{p_1}^{q_1}(S,T)$
\begin{equation}\label{eq3.3}
E\left(\int_{S}^{T}|f(t,M_t)|dt\bigg|{\mathcal{F}_S}\right)\leq N\Vert f\Vert_{\mathbb{L}_{p_1}^{q_1}(S,T)},
\end{equation}
where $N$ depends only on $d$, $K$, $p_1$, $q_1$, $T_0$. Applying \eqref{eq3.3} to $f=|g|^2$ we get
$$E\left(\int_S^T|g(t,M_t)|^2dt\bigg|{\mathcal{F}_S}\right)\leq N\Vert g^2\Vert_{\mathbb{L}_{p/2}^{q/2}(S,T)}= N \Vert g\Vert_{\mathbb{L}_p^q(S,T)}^2.$$
By Lemma \ref{lemm6.1}, for any $\kappa\in[0,\infty)$  we have
$$E\exp(\kappa\int_0^{T_0}|g(t, M_t)|^2dt)\leq N(\kappa, K, d, p, q,T_0, \Vert g\Vert_{\mathbb{L}_p^q(T_0)}),$$
then
\begin{align}\label{exf}
 E\exp(\kappa\int_0^{T_0}|g(t, M_t)|^2dt)\leq N(\kappa, K, d, p, q,T_0, \Vert g\Vert_{\mathbb{L}_p^q(T_0)}).
\end{align}
And also
\begin{align}\label{exb}
E\exp(\kappa\int_0^{T_0}|b(t, M_t)|^2dt)\leq N(\kappa, K, K_0, d, p, q, T_0).
\end{align}
 The integral over $(0,T_0)$ in \eqref{exb} can be replaced with the one over $(0,\infty)$ since $b(t,x)=0$ for $t>T_0$. Thus for any $\kappa\in[0,\infty)$
  \begin{equation}\label{eq3.4}
  E\exp(\kappa\int_0^\infty|b(t, M_t)|^2dt)<\infty,
  \end{equation}
  which  and \eqref{uniell1} implies that for any $c\in[0,\infty)$
  \begin{align}\label{star}
  E\exp(c\int_0^\infty (b^*(\sigma\sigma^*)^{-1}b)(t,M_t)dt)\leq E\exp(\frac{c}{\delta}\int_0^\infty|b(t, M_t)|^2dt)<\infty.
  \end{align}
 For $f\in\mathbb{L}_{p'}^{q'}(S,T)$ with $p'$, $q'\in(1,\infty)$, we can choose $\beta>1$ sufficiently close to $1$ such that $$\frac{d}{p'}+\frac{2}{q'}<\frac{2}{\beta}.$$ By Lemma \ref{lemm3.1} we obtain the existence and uniqueness of $(\mathcal{F}_t)-$adapted process $(X_t)_{t\geq0}$ which satisfies \eqref{solutioninte} and \eqref{eqadd3}.  By Lemma \ref{Girsanov}, we have
   \begin{align}\label{ef} E\int_{S}^{T}|f(t,X_t)|dt=E\int_{S}^{T}\rho |f(t,M_t)|dt&\leq (E\int_{S}^{T}\rho^\alpha dt)^{1/\alpha}(E\int_{S}^{T} |f(t,M_t)|^\beta dt)^{1/\beta}\nonumber\\& \leq (E\int_{0}^{T_0}\rho^\alpha dt)^{1/\alpha}(E\int_{S}^{T} |f(t,M_t)|^\beta dt)^{1/\beta},\end{align}
   where $\alpha$, $\beta>1$ satisfying $1/\alpha+1/\beta=1,$ and
     \begin{align*}
     \rho:=\exp(-\int_0^\infty b^*(\sigma^*)^{-1}(s,M_s)dW_s-\frac{1}{2}\int_0^\infty(b^*(\sigma\sigma^*)^{-1}b)(s,M_s)ds).
     \end{align*}
     Since
     \begin{align}\label{exponential}
     E\rho^\alpha=E\Big[&\Big(\exp(-2\alpha\int_0^\infty b^*(\sigma^*)^{-1}(s,M_s)dW_s-2\alpha^2\int_0^\infty(b^*(\sigma\sigma^*)^{-1}b)(s,M_s)ds)\Big)^{1/2}\nonumber\\
     &\Big(\exp((2\alpha^2-\alpha)\int_0^\infty(b^*(\sigma\sigma^*)^{-1}b)(s,M_s)ds)\Big)^{1/2}\Big],
     \end{align}
     by H\"older's inequality and the fact that exponential martingale is a supermartingale and \eqref{star}, we get
     \begin{align}\label{en}
     E\rho^\alpha \leq N.
     \end{align} Then
   \begin{align*} E\int_{S}^{T}|f(t,X_t)|dt&\leq N(T_0)(E\int_{S}^{T} |f(t,M_t)|^\beta dt)^{1/\beta}\\&\leq N(d,p_1,q_1,K,\Vert b\Vert_{\mathbb{L}_{p}^{q}(T_0)})\Vert f^\beta\Vert_{\mathbb{L}_{p_1}^{q_1}(S,T)}^{1/\beta}\\&=N(d,p_1,q_1,K,\Vert b\Vert_{\mathbb{L}_{p}^{q}(T_0)})\Vert f\Vert_{\mathbb{L}_{\beta p_1}^{\beta q_1}(S,T)} \end{align*}
   for $d/{p_1}+2/{q_1}<2,$ where $p_1=p'/\beta$, $q_1=q'/\beta$. Thus the above estimate implies \eqref{eq3.42}.

  Furthermore,  according to Lemma \ref{Girsanov}  and \eqref{exf},
  \begin{align*}
  E\exp(\kappa\int_0^{T_0}|g(t,X_t)|^2dt)=&E(\rho \exp(\kappa\int_0^{T_0}|g(t,M_t)|^2dt))\nonumber\\\leq & (E\rho^2)^{1/2}(E\exp(2\kappa\int_0^{T_0}|g(t,M_t)|^2dt))^{1/2}<\infty.\end{align*}
   \end{proof}

\iffalse
\begin{remark}\label{rem3.5}
 In the later use of \eqref{expx},  we replace $\tau$ by taking bounded $(F_t)-$stopping times $\xi^n:=\inf\left\{t>0:(t,X_t)\notin Q^n\right\}$, $n\in\mathbb{N}$. Precisely, if $|bI_{Q^n}|$, $|gI_{Q^n}|\in \mathbb{L}_p^q$ with $p,q\in(2,\infty)$ and $d/p+2/q<1$, for  $t\in[0,\xi^n)$, the process $(X_t)_{t\geq0}$ satisfies the equation
 $$X_t=x+\int_0^tbI_{Q^n}(s,X_s)ds+\int_0^t\sigma(s,X_s)dW_s,$$ in which case the requirements in Assumption \ref{ass3.2} are satisfied. Then
 \begin{align}\label{stopnvkv}
 E\exp(\kappa\int_0^{T\wedge\xi^n}|g(t,X_t)|^2dt)=E\exp(\kappa\int_0^\infty|g(t,X_t)I_{t<\xi^n}|^2dt)<\infty.
 \end{align}

\end{remark}
\fi
\begin{lemma}\label{lemm3.6}
 Let $b^{(i)}(t,x)$, $i=1,2$ satisfy Assumption \ref{ass3.2} and let $|b^{(1)}(t,x)-b^{(2)}(t,x)|\leq \overline{b}(t,x)$, where $\overline{b}$ also satisfies Assumption \ref{ass3.2}. Let $(X_t^{(i)}, W_t^{(i)})_{t\geq 0}$ satisfy
 $$X_t^{(i)}=x+\int_0^tb^{(i)}(s,X_s^{(i)})ds+\int_0^t\sigma(s,X_s^{(i)})dW_s^{(i)},\quad t\geq 0.$$
 Then for any bounded Borel functions $f^{(i)}$, $i=1,2$ given on $\mathcal{C}:=\mathcal{C}([0,\infty),\mathbb{R}^d)$ we have
 \begin{equation}\label{eq3.5}|Ef^{(1)}(X_\cdot^{(1)})-Ef^{(2)}(X_\cdot^{(2)})|\leq N(E|f^{(1)}(M_\cdot)-f^{(2)}(M_\cdot)|^2)^{1/2}+N\sup_{\mathcal{C}}|f^{(1)}|\Vert \overline{b}\Vert_{\mathbb{L}_p^q}
 \end{equation}
 where $M_t=\int_0^t\sigma(s,M_s)dW_s$, $t\geq 0$, and $N$ is a constant independent of $f$.
\end{lemma}

\begin{proof}
 According to Lemma \ref{Girsanov}, we know that
 $$Ef^{(2)}(X_\cdot^{(2)})=Ef^{(2)}(X_\cdot^{(1)})\overline{\rho}_\infty,$$
 where for $t\geq 0$, $\Delta b(t,X_t^{(1)}):=b^{(2)}(t,X_t^{(1)})-b^{(1)}(t,X_t^{(1)})$ and
 $$\overline{\rho}_\infty:=\exp(\int_0^\infty\Delta b^*(\sigma^*)^{-1}(s,X_s^{(1)})dW_s^{(1)}-\frac{1}{2}\int_0^\infty(\Delta b^*(\sigma\sigma^*)^{-1}\Delta b)(s,X_s^{(1)})ds),$$
 also $E\overline{\rho}_\infty=1 $ by applying \eqref{expx} and the fact that $\Delta b(t,x)=0$ for $t>T_0$ and \eqref{uniell1}.
Hence the left-hand side of \eqref{eq3.5} is less than
$$E|f^{(1)}-f^{(2)}|(X_\cdot^{(1)})\overline{\rho}_\infty+\sup_\mathcal{C}|f^{(1)}| E|\overline{\rho}_\infty-1|=:I_1+I_2\sup_\mathcal{C}|f^{(1)}|.$$
Also we have that all moments of the exponential martingale $$\overline{\rho}_t=\exp(\int_0^t\Delta b^*(\sigma^*)^{-1}(s,X_s^{(1)})dW_s^{(1)}-\frac{1}{2}\int_0^t(\Delta b^*(\sigma\sigma^*)^{-1}\Delta b)(s,X_s^{(1)})ds)$$ are finite by the same argument as getting \eqref{en} in Lemma \ref{lemm3.4}. Hence we get
\begin{align}\label{I1r}
I_1^{3/2}\leq NE|f^{(1)}-f^{(2)}|^{3/2}(X_\cdot^{(1)})
\end{align}
and right hand side of \eqref{I1r} is controled by the first term on the right hand side of  \eqref{eq3.5} by a similar argument as dealing with \eqref{ef} in Lemma \ref{lemm3.4}. To estimate $I_2$, we use It\^o's formula to get for any $T\in[0,\infty)$,
$$\overline{\rho}_T=1+\int_0^T(\Delta b^*(\sigma^*)^{-1})(s,X_s^{(1)})\overline{\rho}_sdW_s^{(1)}.$$
It follows that for any $\beta>1$
\begin{align}\label{ad}
I_2^2\leq &  E|\overline{\rho}_{T_0}-1|^2\leq  E\int_0^{T_0}(\Delta b^*(\sigma\sigma^*)^{-1}\Delta b)(s,X_s^{(1)})\overline{\rho}_s^2ds\nonumber\\
\leq  &N(\int_0^{T_0}E\overline{\rho}_s^{2\beta/(\beta-1)}ds)^{1-1/\beta}(E\int_0^{T_0}|\Delta{b}(s,X_s^{(1)})|^{2\beta}ds)^{1/\beta} \nonumber\\
\leq  &N(\int_0^{T_0}E\overline{\rho}_s^{2\beta/(\beta-1)}ds)^{1-1/\beta}(E\int_0^{T_0}|\overline{b}(s,X_s^{(1)})|^{2\beta}ds)^{1/\beta}.\end{align}
To estimate the second factor of the the right hand of \eqref{ad} we use Lemma \ref{lemm3.4} with $\beta>1$ close to 1 such that $2/q+d/p<1/\beta$. The first factor of the the right hand of \eqref{ad} is controlled by means of $ E \overline{\rho}_{T_0}^{2\beta/(\beta-1)}$. Thus the result follows.
\end{proof}

\subsection{Proof of Theorem \ref{mainadd1}}

\indent Now we are going to prove the maximal local well-posedness result on an arbitrary domain $Q\subset \mathbb{R}_+\times\mathbb{R}^{d}$ by applying the localization technique, which is a modification of the proof of Theorem 1.3 in \cite{Zhang2011}. Furthermore we will prove the continuity of the solution on the domain $Q'=Q\cup \partial$, especially around the boundary $\partial Q'$.

\begin{proof}[Proof of Theorem \ref{mainadd1}]
\indent By Lemma \ref{Urysohn}, for each $n\in \mathbb{N}$, we can find a nonnegative smooth function $\chi_n(t,x)\in[0,1]$ in $\mathbb{R}^{d+1}$  such that $\chi_n(t,x)=1$ for all $(t,x)\in Q^n$ and $\chi_n(t,x)=0$ for all $(t,x)\notin Q^{n+1}$. For any $s,x\in Q$, let
$$b_s^n(t,x):=\chi_n(t+s,x)b(t+s,x)$$
and
$$\sigma_s^n(t,x):=\chi_{n+1}(t+s,x)\sigma(t+s,x)+(1-\chi_n(t+s,x))(1+\sup_{(t+s,x)\in Q^{n+2}}|\sigma(t+s,x)|)\mathbb{I}_{d\times d}.$$
By Lemma \ref{lemm3.1} there exists a unique $(\mathcal{F}_t)$-adapted continuous solution $(X_t^n)_{t\geq0}$ satisfying
\begin{align}\label{xn}
X_t^n=x+\int_0^tb_s^n(r,X_r^n)dr+\int_0^t\sigma_s^n(r,X_r^n)dW_r, \quad \forall t\in [0,\infty), a.s.
\end{align}
More precisely, for condition (i) in Lemma \ref{lemm3.1}, for any $(t,x)\in[0,\infty)\times\mathbb{R}^d$,
$$
|b_s^n(t,x)|\leq|(bI_{Q^{n+2}})(s+t,x)|,
$$
\begin{align*}
|\nabla\sigma_s^n(t,x)|&\leq|(\nabla\chi_{n+1}\sigma)(t+s,x)|+|(\chi_{n+1}\nabla\sigma)(t+s,x)|+c|\nabla\chi_{n}(t+s,x)|\\&\leq |(\nabla\chi_{n+1}\sigma  I_{Q^{n+2}})(t+s,x)|+|(\nabla\sigma I_{Q^{n+2}})(t+s,x)|+c|\nabla\chi_{n}(t+s,x)|,
\end{align*}
which means that we can take $p:=p_{n+2}$, $q:=q_{n+2}$.
The continuity condition in Lemma \ref{lemm3.1} (ii)  obviously holds. Further there exist constants $K(n)$ and $\delta_{n+1}$ such that for all $(t,x)\in \mathbb{R}_+\times\mathbb{R}^{d}$ and $\lambda\in\mathbb{R}^d,$
$$|(\sigma_s^n)^*(t,x)\lambda|^2\leq|(\sigma^* I_{Q^{n+2}}+(1+\sup_{(s+t,x)\in Q^{n+2}}|\sigma^*(s+t,x)|)\mathbb{I}_{d\times d})(s+t,x)\lambda|^2\leq K(n)|\lambda|^2,$$
and
\begin{align*}|(\sigma_s^n)^*(t,x)\lambda|^2&\geq|(\sigma^* I_{Q^{n+1}}+I_{(Q^{n+1})^\mathsf{c}\cap Q^{n+2}}\\&\quad\quad\quad+I_{(Q^{n+2})^\mathsf{c}}(1+\sup_{(s+t,x)\in Q^{n+2}}|\sigma^*(s+t,x)|)I_{d\times d})(s+t,x)\lambda|^2\\
&\geq(\delta_{n+1}\wedge 1)|\lambda|^2.
\end{align*}
Thus equation \eqref{xn} satisfies conditions (i) and (ii) in Lemma \ref{lemm3.1}.
%The non-explosion is from Girsanov transformation and Krylov estimation. In fact from Krylov estimation we get for all $t\in[0,\infty)$
%$$E\exp(\frac{1}{2}\int_0^t ((b^n)^*(\sigma^n(\sigma^n)^*)^{-1}b^n)(s,Y_s)ds)<\infty,$$
%where $Y_t^n=:x+\int_0^t\sigma^n(s+r,Y_r^n)dW_r$. Then Girsanov Theorem tells us
%$$E|X_{t}^n|=E\exp(-\int_0^t ((b^n)^*{(\sigma^n)^*}^{-1})(s,Y_s)dWs-\frac{1}{2}\int_0^t ((b^n)^*(\sigma^n(\sigma^n)^*)^{-1}b^n)(s,Y_s)ds)|Y_{t}^n|.$$
%Since $Y_t^n$ is non-explosive from boundedness of $\sigma^n$, H\"older inequality and the property of exponential martingale imply that $X_t^n$ is non-explosive.\\
For $n\geq k$, define
$$\tau_{n,k}:=\inf\left\{t\geq0: z_t^n:=(s+t,X_t^n)\notin Q^k\right\},$$
then it is easy to see that $X_t^n$, $X_t^k$, $t\geq0$, satisfy
\begin{align*}
X_{t\wedge\tau_{n,k}}=&x+\int_0^{t\wedge\tau_{n,k}}b_s^k(r,X_r)dr+\int_0^{t\wedge\tau_{n,k}}\sigma_s^k(r,X_r)dW_r,\quad a.s.
\end{align*}
By the local uniqueness of the solution in Lemma \ref{lemm3.1}, we have
\begin{align*}
P\left\{\omega:X_{t}^n(\omega)=X_{t}^k(\omega),\forall t\in[0,\tau_{n,k}(\omega))\right\}=1,
\end{align*}
 which implies $\tau_{k,k}\leq\tau_{n,k}\leq\tau_{n,n}$ $a.s.$. Thus if we take $\xi_{k}:=\tau_{k,k}$, then $\xi_k$ is an increasing sequence of stopping times, and
 \begin{align*}
 P\left\{\omega:X_{t}^n(\omega)=X_{t}^k(\omega),\forall t\in[0,\xi_k(\omega))\right\}=1.
 \end{align*}
 Now for each $k\in\mathbb{N}$, the definitions
  \begin{align*}
  X_t(\omega):=X_t^k(\omega) \text{ for } t<\xi_k,\quad
  \xi:=\lim_{k\rightarrow\infty}\xi_k,
  \end{align*}
  and
  \begin{align}\label{cz}
 z_t=(s+t,X_t),\quad t<\xi,\quad z_t=\partial,\quad \xi\leq t<\infty,
  \end{align}
  make sense almost surely. We may throw the set of $\omega$ where the above definitions do not make sense and work only on the remaining part of $\Omega$.
   Then $(X_t)_{t\geq0}$ satisfies the SDE \eqref{eq1add} and $\xi$ is the related explosion time. \\
   \indent The last thing is to prove that $(z_t)_{t\geq0}$ from \eqref{cz} is continuous on $Q'$. Since $z_t$ coincides with $(t,X_t^n)$ before $\xi_n$, the continuity  of $z_t$ before $\xi_n$ follows from the continuity of $(t,X_t^n)$, which can be obtained by Lemma \ref{lemm3.1}. So we only need to show that $z_t$ is left continuous at $\xi$ $a.s.$. The argument essentially follows from \cite{KR}.
   {  We first show that $(z_t)_{t\geq0}$ has  strong Markov property. We use $P_{s,x}^n$ to denote the distribution of process $(z_t^n)_{t\geq0}=(z_t^n(s,x))_{t\geq0}:=(s+t,X_t^n(0,x))_{t\geq0}$ on $\mathcal{C}([0,\infty),\mathbb{R}^{d+1})$, where $(X_t^n(s,(0,x)))_{t\geq0}$ means the solution $(X_t^n)_{t\geq0}$ defined above with initial point $(0,x)\in \mathbb{R}^{d+1}$.  $E_{s,x}^n$ denotes the expectation corresponding to $P_{s,x}^n$.
   %From now on the process $(X_t^n)_{t\geq0}$ we always mean  $(X_t^n(s,(0,x)))_{t\geq0}$. 
   The following argument  is based on Proposition 4.3.3 of \cite{LWMR}.\\
    \indent Define the space $\mathbb{W}_0:=\left\{w\in\mathcal{C}(\mathbb{R}_+,\mathbb{R}^d)|w(0)=0\right\}$ equipped with the supremum norm and Borel $\sigma-$algebra $\mathcal{B}(\mathbb{W}_0)$, the class $\mathcal{E}$ collects all the maps $F:\mathbb{R}^d\times\mathbb{W}_0\rightarrow\mathcal{C}(\mathbb{R}_+,\mathbb{R}^d)$ such that for every probability measure $\mu$ on $(\mathbb{R}^d,\mathcal{B}(\mathbb{R}^d))$ there exists a $\overline{\mathcal{B}(\mathbb{R}^d)\times\mathcal{B}(\mathbb{W}_0)}^{\mu\times P^{W}}$ $/\mathcal{B}(\mathbb{R}^d)$ measurable map $F_\mu:\mathbb{R}^d\times\mathbb{W}_0\rightarrow\mathcal{C}(\mathbb{R}_+,\mathbb{R}^d)$ such that for $\mu-a.e. x\in\mathbb{R}^d$ we have $F(x,w)=F_{\mu}(x,w)$ for $P^W-a.e.$ $w\in\mathbb{W}_0$. Here $\overline{\mathcal{B}(\mathbb{R}^d)\times\mathcal{B}(\mathbb{W}_0)}^{\mu\times P^{W}}$ means the completion of $\mathcal{B}(\mathbb{R}^d)\times\mathcal{B}(\mathbb{W}_0)$ with respect to $\mu\times P^W$, and $P^W$ denotes the distribution of the standard $d$-dimensional Wiener process $(W_t)_{t\geq0}$ on $(\mathbb{W}_0,\mathcal{B}(\mathbb{W}_0))$.  For each $n\in\mathbb{N}$, since we already have the pathwise uniqueness and existence of strong solution $(X_t^n)_{t\geq0}$ to \eqref{xn},  by  applying Theorem E.8 in \cite{LWMR},  we obtain that there exists a map $F\in\mathcal{E}$ such that for $u\leq t$ we have  $X_t^n(s,(0,x))(\omega)=F_{P\circ({X_u^n(s,(0,x))})^{-1}}({X_u^n(s,(0,x))}(\omega),(W_\cdot-W_u)(\omega))(t)$ for $P-a.e.$ $\omega\in\Omega$.
     %where $X_t^n(0,x)$ means the solution $(X_t^n)_{t\geq0}$ to \eqref{xn} with initial point $x\in\mathbb{R}^d$ at time $0$.  
     Then for every bounded measurable function $f$ and all $u,t\in[0,\infty)$ with $u\leq t$ we have for $P-a.e.$ $\omega\in\Omega$
     \begin{align}\label{xnmarkov}
     E[f(X_t^n(s,(0,x)))|\mathcal{F}_u](\omega)&= E[f(F_{P\circ({X_u^n(s,(0,x))})^{-1}}({X_u^n(s,(0,x))}(\omega),W_\cdot-W_u)(t))]\nonumber\\&= E[f(F_{\delta_{X_u^n(s,(0,x))(\omega)}}({X_u^n(s,(0,x))(\omega)},W_\cdot-W_u)(t))]\nonumber\\&=E[f(X_t^n(s,(u,X_u^n(s,(0,x))))(\omega))],
     \end{align}
     which shows the Markov property of the process $(X_t^n)_{t\geq0}$. Here $X_t^n(s,(u,X_u^n(s,(0,x))))$ means the solution $(X_t^n)_{t\geq0}$ to \eqref{xn} with starting point $(u,X_u^n(s,(0,x)))\in\mathbb{R}^{d+1}$. Combining with the Feller property of $(X_t^n)_{t\geq0}$ yielding from the second statement of Lemma \ref{lemm3.1} and well known results about Markov processes (see e.g. \cite[Theorem 16.21]{LB}), we get that $(X_t^n)_{t\geq0}$ is a strong Markov process.
     %%%%%%%%%%%%%%%%%%%%%%%%%%%%%%%%%%%%%%%%%%%%%%%%%%%%%%%%%%%%%%%%%%%%%%%%%%%%%%%%%%%%%%%%%%%%%
      \iffalse  For the process $(z_t^n)_{t\geq0}$, for every bounded Borel measurable function $G$ defined on $\mathbb{R}^{d+1}$, we have
     \indent Since the process $(X_t^n)_{t\geq0}$ is continuous, for any $f\in\mathcal{C}(\mathbb{R}^{d+1})$ and $(u,x),(s,y)\in\mathbb{R}^{d+1}$,
     \begin{align*}
      |E_{u,x}^nf(z_t^n)-E_{s,y}^nf(z_t^n)|\leq& |E_{u,x}^nf(z_t^n)-E_{u,y}^nf(z_t^n)|+|E_{u,y}^nf(z_t^n)-E_{s,y}^nf(z_t^n)|\\
      =&|Ef(u+t,X_t^n(0,x))-Ef(u+t,X_t^n(0,y))|\\&+|Ef(u+t,X_t^n(0,y))-Ef(s+t,X_t^n(0,y))|,
      \end{align*}
      by the continuity of $f$ with respect to time and space variable and the strong Feller property of $(X_t^n)_{t\geq0}$ getting from the second statement of Lemma \ref{lemm3.1}, we obtain that $E_{u,x}^nf(z_t^n)$ is continuous with respect to $(u,x)$, which implies that $(t,X_t^n)_{t\geq0}=(z_t^n)_{t\geq0}$ is a Feller  process.  By well known results about Markov processes (see e.g. \cite[Theorem 16.21]{LB}), we get that $(z_t^n)_{t\geq0}$ is a strong Markov process.\fi
      %%%%%%%%%%%%%%%%%%%%%%%%%%%%%%%%%%%%%%%%%%%%%%%%%%%%%%%%%%%%%%%%%%%%%%%%%%%%%%%5
  Now we are going to prove that $(z_t^n)_{t\geq0}$ is a strong Markov process.  Observing that for $u\geq0$, $(\hat W_t)_{t\geq0}:=(W_{t+u}-W_u)_{t\geq0}$ is still a Brownian motion. For any $(s,x)\in Q$, and for any Borel bounded function $f$ on $\mathbb{R}^{d+1}$, % which means that $f(t,\cdot)$ is  Borel bounded function on $\mathbb{R}^{d}$ for any $t\in\mathbb{R}$,
   by \eqref{xnmarkov}, we have for any $u, t\geq 0$,  $P-a.e.$ 
   \begin{align*}
  X_{t+u}^n&(s,(u,X_u^n(s,(0,x))))\\&=X_u^n(s,(0,x))+\int_{u}^{u+t}\sigma_s^n(r,X_r^n(s,(0,x)))d(W_r-W_u)
  \\&\quad\quad\quad\quad\quad\quad\quad\quad\quad\quad\quad\quad\quad\quad\quad\quad+\int_{u}^{u+t}b_s^n(r,X_r^n(s,(0,x)))dr
  \\&=X_u^n(s,(0,x))+\int_0^t\sigma_{s}^n(r+u,X_{u+r}^n(s,(0,x)))d\hat W_{r}
  \\&\quad\quad\quad\quad\quad\quad\quad\quad\quad\quad\quad\quad\quad\quad\quad\quad+\int_{0}^{t}b_{s}^n(r+u,X_{r+u}^n(s,(0,x)))dr,
  %\\&=X_t^n(s+u,(0,X_u^n(s,(0,x)))),
  %X_u^n(s,(0,x))+\int_0^t\sigma^n(u+r,(u,X_{u+r}^n(s,(0,x))))dW_{u+r}
  \end{align*}
  and 
  \begin{align*}
  X_t^n&(s+u,(0,X_u^n(s,(0,x))))\\&=X_u^n(s,(0,x))+\int_0^t\sigma_s^n(u+r,X_{r}^n(u+s,(0,X_u^n(s,(0,x)))))d\hat W_{r}
  \\&\quad\quad\quad\quad\quad\quad\quad\quad\quad\quad\quad\quad\quad\quad\quad\quad+\int_{0}^{t}b_s^n(u+r,X_{r}^n(u+s,(0,X_u^n(s,(0,x)))))dr
  \\&=X_u^n(s,(0,x))+\int_0^t\sigma_{s+u}^n(r,X_{r}^n(u+s,(0,X_u^n(s,(0,x)))))d\hat W_{r}
  \\&\quad\quad\quad\quad\quad\quad\quad\quad\quad\quad\quad\quad\quad\quad\quad\quad+\int_{0}^{t}b_{s+u}^n(r,X_{r}^n(u+s,(0,X_u^n(s,(0,x)))))dr.
  \end{align*}
  Since $\sigma_s^n(u+r,\cdot)=\sigma_{s+u}^n(r,\cdot)$, and $b^n_s(u+r,\cdot)=b^n_{s+u}(r,\cdot)$, by the  pathwise uniqueness of the  the following equation
  $$dX_t=\sigma_{s+u}^n(t,X_t)d\hat W_t+b_{s+u}^n(t,X_t)dt,\quad X_0=X_u^n(s,(0,x)),$$  we have for arbitrary Borel bounded function $h$ on $\mathbb{R}^d$, $Eh( X_{t+u}^n(s,(u,X_u^n(s,(0,x)))))=Eh(X_t^n(s+u,(0,X_u^n(s,(0,x)))))$. Hence for $P-a.e.$ $\omega\in\Omega$,
  \begin{align*}
  E[f(z_{t+u}^n(s,x))|\mathcal{F}_u](\omega)&= E[f(s+t+u,X_{t+u}^{n}(s,(0,x)))|\mathcal{F}_u](\omega)\\&=E[f(s+t+u,X_{t+u}^n(s,(u,X_u^n(s,(0,x))))(\omega))]
  \\&=E[f(s+t+u,X_t^n(s+u,(0,X_u^n(s,(0,x)))))(\omega))]
  \\&=E_{z_u^n(s,x)(\omega)}^nf(z_{t}^n).
  \end{align*}
  %The above equalities hold since by \eqref{xnmarkov} we have $P-a.e.$
  So $(z_t^n)_{t\geq0}$ is a Markov process. Furthermore, for any $(s,x)\in Q$, by applying Ito's formula to process $X_r^n(s,(0,x))$, we get that $u_s^n(t,x)=Ef(X_t^n(s,(0,x)))$ is the solution to the following equation
    \begin{equation}\label{sn}
     \left\{
     \begin{aligned}
    &D_ru_s^n(r,x)=\frac{1}{2}\sum_{i,j=1}^da_{s,ij}^n(r,x)\partial_{i}\partial_ju_s^n(r,x)+b_s^n(r,x)\cdot\nabla u_s^n(r,x)\text{ on } (0,\infty)\times\mathbb{R}^d, \\
    & u_s^n(0,x) = f(x), \\
    \end{aligned}
    \right.
    \end{equation}
    with  $(a_{s,ij}^n)_{1\leq i,j\leq d}=\sigma_s^n\cdot(\sigma_s^n)^*$,
   and Borel bounded continuous function $f$ on $\mathbb{R}^{d}$. Let $u^n(t,x)$ be the solution to the following equation
     \begin{equation}\label{-s}
     \left\{
     \begin{aligned}
    &D_ru^n(r,x)=\frac{1}{2}\sum_{i,j=1}^da_{ij}^n(r,x)\partial_{i}\partial_ju^n(r,x)+b^n(r,x)\cdot\nabla u^n(r,x)\text{ on } (s,\infty)\times\mathbb{R}^d, \\
    & u^n(s,x) = f(x), \\
    \end{aligned}
    \right.
    \end{equation}
    with  $(a_{ij}^n)_{1\leq i,j\leq d}=\sigma^n\cdot(\sigma^n)^*$, and $\sigma^n$ and $b^n$ are defined as following
      $$b^n(r,x):=b_0^n(r,x),\quad \sigma^n(r,x):=\sigma_0^n(r,x).$$
Then it is easy to see that $u^n(s+t,x)$  also satisfies \eqref{sn}, which by using uniqueness of solution to \eqref{sn} implies $u_s^n(t,x)=u^n(s+t,x)=Ef(X_t^n(s,(0,x)))$. By Remark 10.4 \cite{KR} (or see Theorem 3.1 \cite{XXZZ}), we know that the unique solution $u^n(t,x)$ to the above equation \eqref{-s} is continuous on $(t,x)\in[s,\infty)\times\mathbb{R}^{d}$, which yields the continuity of $Ef(X_t^n(s,(0,x)))$ with respect to $(s,x)\in[0,\infty)\times\mathbb{R}^{d}$ for any $t\in[0,\infty)$.  Then the second statement of Lemma \ref{lemm3.1} and dominated convergence theorem imply that for any Borel bounded continuous function $g$ on $\mathbb{R}^{d+1}$, and for any $(s,x)\in[0,\infty)\times\mathbb{R}^d$
  \begin{align*}
  \lim_{(u,y)\rightarrow(s,x)}E_{u,y}^ng(z_t^n)&=  \lim_{(u,y)\rightarrow(s,x)}Eg(u+t,X^n_t(u,(0,y)))\\&= \lim_{(u,y)\rightarrow(s,x)}\Big(Eg(u+t,X^n_t(u,(0,y)))-Eg(u+t,X^n_t(u,(0,x)))\Big)
  \\&\quad\quad+\lim_{(u,y)\rightarrow(s,x)}\Big(Eg(u+t,X^n_t(u,(0,x)))-Eg(s+t,X^n_t(u,(0,x)))\Big)
  \\&\quad\quad+\lim_{(u,y)\rightarrow(s,x)}Eg(s+t,X^n_t(u,(0,x)))
  \\&\leq \lim_{(u,y)\rightarrow(s,x)}C_t\Vert g(u+t,\cdot)\Vert_\infty|x-y|+ Eg(s+t,X^n_t(s,(0,x)))
  \\&= Eg(t+s,X^n_t(s,(0,x)))=E_{s,x}^ng(z_t^n).
  \end{align*}
  It shows that $(z_t^n)_{t\geq0}$ also has  Feller property, hence $(z_t^n)_{t\geq0}$ is a  strong Markov process. Then for any $(s,x)\in Q$, for any $(\mathcal{F}_t)$-adapted stopping time $\eta$  and for any Borel bounded function $f$ on $\mathbb{R}^{d+1}$,
     \begin{align}\label{mr}
     E_{s,x}f(z_{\eta+t})=f(\partial)+E_{s,x}(f(z_{\eta+t})-f(\partial))I_{\xi>\eta+t}.
     \end{align}
     Since
     \begin{align*}
     E_{s,x}f(z_{\eta+t})I_{\xi>\eta+t}&=\lim_{n\rightarrow\infty}E_{s,x}f(z_{\eta+t})I_{\xi_n\geq\eta+t}\\
     &=\lim_{n\rightarrow\infty}E_{s,x}^nf(z_{\eta+t}^n)I_{\xi_n\geq\eta+t}I_{\xi_n\geq\eta}
     \\&=\lim_{n\rightarrow\infty}E_{s,x}^nf(\eta+t,X_{\eta+t}^n)I_{\xi_n\geq\eta+t}I_{\xi_n\geq\eta},
     \end{align*}
     and $\{\xi_n\geq\eta\}\subset\mathcal{F}_\eta$, by the strong Markov property of $(z_t^n)_{t\geq0}$, we get
     \begin{align*}
     \lim_{n\rightarrow\infty}E_{s,x}^nI_{\xi_n\geq\eta}E_{(\eta,X_\eta^n)}^nf(t,X_{t}^n)I_{\xi_n\geq\eta}=
     &\lim_{n\rightarrow\infty}E_{s,x}^nI_{\xi_n\geq\eta}E_{z_\eta^n}^nf(z_{t}^n)I_{\xi_n\geq\eta}
     %\\=&\lim_{n\rightarrow\infty}E_{s,x}I_{\xi_n\geq\eta}E_{z_\eta}f(z_{t})I_{\xi_n\geq\eta}
     \\=&E_{s,x}I_{\xi>\eta}E_{(\eta,X_\eta)}f(t,X_t)I_{\xi>\eta}.
     \\=&E_{s,x}I_{\xi>\eta}E_{z_\eta}f(z_t)I_{\xi>\eta}\end{align*}
     Then \eqref{mr} yields
     \begin{align}\label{me}E_{s,x}f(z_{\eta+t})=E_{s,x}E_{z_\eta}f(z_t).\end{align}
     We can find that \eqref{me} also holds if we replace $(s,x)$ with $\partial$. Hence we get the strong Markov property of the process $(z_t)_{t\geq0}$.\\
}
     \indent In the following we will prove another two auxiliary lemmas in order to show that our solution does not bounce back deep into the interior of $Q$ from near $\partial Q$ too often on any finite interval of time, which is crucial for us to prove the desired continuity.  By shifting the origin in $\mathbb{R}^{d+1}$, without losing generality, we assume $(s,x)=(0,0)$.

 \begin{lemma}\label{ossilaton1}
  For arbitrary $n\geq0$, define $\nu_0=0$,
 \begin{align}\label{bounce}
 \mu_k=\inf\left\{t\geq\nu_k:(t,X_t)\notin Q^{n+1}\right\},\quad \nu_{k+1}=\inf\left\{t\geq\mu_k:(t,X_t)\in {\overline{ Q^n}}\right\}.
 \end{align}
 Then for any $S\in(0,\infty)$ there exists a constant $N$, depending only on $d$, $p$, $q$, $S$, $\Vert bI_{Q^{n+1}}\Vert_{\mathbb{L}_p^q}$, $\sup_{(t,x)\in Q^{n+1}}|\sigma(t,x)|$, and the diameter of $Q^{n+1}$, such that
 \begin{align*}
 \sum_{k=0}^\infty(E|X_{S\wedge\mu_k}-X_{S\wedge\nu_k}|^2)^2\leq N,\quad \sum_{k=0}^\infty(E|S\wedge\mu_k-S\wedge\nu_k|^2)^2\leq S^4.
 \end{align*}
 \end{lemma}
 \begin{proof}
 We have $E|X_{S\wedge\mu_k}-X_{S\wedge\nu_k}|^2\leq 2I_k+2J_k$, where
 \begin{align*}
 I_k:=E|\int_{S\wedge\nu_k}^{S\wedge \mu_k}\sigma(s,X_s)dW_s|^2,\quad J_k:=E|\int_{S\wedge\nu_k}^{S\wedge \mu_k}b(s,X_s)ds|^2.
 \end{align*}
 Observe that on the set $\left\{S\wedge\nu_k<S\wedge\mu_k\right\}$ we have $S\wedge\nu_k=\nu_k$ and $(\nu_k,X_{\nu_k})\in\overline{ Q^n}\subset Q^{n+1}$. Furthermore, $(t,X_t)\in Q^{n+1}$ for $S\wedge\nu_k<t<S\wedge\mu_k,$ and we have
 \begin{align*}
 E|\int_{S\wedge\nu_k}^{S\wedge \mu_k}\sigma(s,X_s)dW_s|^2\leq\sum_{i,j=1}^dE|\int_{S\wedge\nu_k}^{S\wedge \mu_k}\sigma_{ij}^2(s,X_s)ds|\leq Cd^2E|S\wedge \mu_k-S\wedge\nu_k|,\\
 I_k^2\leq Cd^4E|S\wedge \mu_k-S\wedge\nu_k|^2=:Cd^4\bar I_k\leq Cd^4SE|S\wedge \mu_k-S\wedge\nu_k|, \quad\\
 \sum_{k=0}^\infty (E|\int_{S\wedge\nu_k}^{S\wedge \mu_k}\sigma(s,X_s)dW_s|^2)^2\leq Cd^4S^2,\quad\sum_{k=0}^\infty (\bar I_k)^2\leq (\sum_{k=0}^\infty\bar I_k)^2\leq S^4.
 \end{align*}
 Moreover, by H\"older's inequality we have
 \begin{align*}
 J_k\leq E|S\wedge\mu_k-S\wedge\nu_k|\int_{S\wedge\nu_k}^{S\wedge\mu_k}|b(s,X_s)|^2ds, \quad J_k^2\leq \bar I_k\bar J_k,
 \end{align*}
 where
 \begin{align*}
 \bar J_k:=E(\int_{S\wedge\nu_k}^{S\wedge\mu_k}|b(s,X_s)|^2ds)^2.
 \end{align*}
 Let $\tau_n:=\inf\left\{t\geq0:z_t\notin Q^{n}\right\}$. {By the strong Markov property of $(z_t)_{t\geq0}$}, it follows that
 \begin{align}\label{J}
 \bar J_k\leq \sup_{(s,x)\in Q^{n+1}}E_{s,x}(\int_0^{S\wedge\tau_{n+1}}|b(s+t,X_t)|^2dt)^2=\sup_{(s,x)\in Q^{n+1}}E_{s,x}(\int_0^{S}|bI_{Q^{n+1}}(s+t,X_t)|^2dt)^2,
 \end{align}
 Since for $t\leq\tau_{n+1}$, $X_t=X_t^{n+1}$,  we see that the second right part of \eqref{J} will not change if we change arbitrarily $b$ outside of $Q^{n+1}$ only preserving the property that new $b$ belongs to $\mathbb{L}_p^q$. We choose to let $b$ be zero outside of $Q^{n+1}$ and then get the desired estimate from \eqref{expx}. The lemma is proved.\\
 \end{proof}
Following the same argument in \cite[Corollary 4.3]{KR} and use Lemma \ref{ossilaton1} we get the following result. In order to reduce duplicate, the proof is omitted.
 \begin{lemma}\label{ossilation2}
 We say that on the time interval $[\nu_k,\mu_k]$ the trajectory $(t,X_t)_{t\geq0}$ makes a run from $\overline{ Q^n}$ to $({Q^{n+1}})^\mathsf{c}$ provided that $\mu_k<\infty$. Denote by $\nu(S)$ the number of runs which $(t,X_t)_{t\geq0}$ makes from $\overline{ Q^n}$ to $Q^{n+1}$ before time $S$. Then for any $\alpha\in[0,1/2)$, $E\nu^\alpha(S)$ is dominated by a constant $N$, which depends only on $\alpha$, $d$, $p$, $q$, $S$, $\Vert bI_{Q^{n+1}}\Vert_{\mathbb{L}_p^q}$, $\sup_{(t,x)\in Q^{n+1}}|\sigma(t,x)|$, the diameter of $Q^{n+1}$, and the distance between the boundaries of $Q^n$ and $Q^{n+1}$.
 \end{lemma}

 \indent Now we go back to prove that $z_t$ is left continuous at $\xi$ $a.s.$.  We denote $\nu_k(S)$  the number of runs of $z_t$ from $\overline{Q^k}$ to $({Q^{k+1}})^\mathsf{c}$ before $S\wedge\xi$. For $n>k+1$ obviously, $\nu_k(S\wedge\xi_n)$ is also the number of runs that $(t,X_t^n)_{t\geq0}$ makes from $\overline{ Q^k}$ to $({Q^{k+1}})^\mathsf{c}$ before $S\wedge\xi_n$, since $(t,X_t)$ coincides with $(t,X_t^n)$ before $\xi_n$. $\nu_k(S\wedge\xi_n)$ increase if we increase the time interval to $S$. By Lemma \ref{ossilation2} $E\nu_k^{1/4}(S\wedge\xi_n)$ is bounded by a constant independent of $n$. By Fatou's Lemma $E\nu_k^{1/4}(S\wedge\xi)$ is finite. In particular, on the set $\left\{\omega:\xi(\omega)<\infty\right\}$ $ a.s.$  we have $\nu_k(\xi)<\infty$. The latter also holds on the set $\left\{\omega:\xi(\omega)=\infty\right\}$  because $z_t$ is continuous on $[0,\xi)$ and $Q^k$ is bounded. Thus $\nu_k(\xi)<\infty$ $a.s.$ for any $k$. Since $(\xi^n, X_{\xi^n}^n)\in\partial Q^n$ we conclude that $a.s$. there can exist only finitely many $n$ such that $z_t$ visits $\overline{ Q^k}$ after exiting from $Q^n$. This is the same as to say that $z_t\rightarrow \partial$ as $t\uparrow\xi$ $a.s.$.\\

 \indent About the uniqueness, if there is another continuous $(\mathcal{F}_t)$-adapted $Q'-$valued solution $(z_t')_{t\geq0}=(s+t,X_t')_{t\geq0}$ to the SDE \eqref{eq1add} with explosion time $\xi'$, and for $t<\xi'$ it is $Q-$valued. Then for any $n\geq 1$
 \begin{align}\label{taun}
 \tau^n(X_\cdot'):=\inf\left\{t\geq 0:(s+t,X_t')\notin Q^n\right\}<\xi'
 \end{align}
 and
 \begin{align}\label{barxi}
 \bar \xi:=\lim_{n\rightarrow\infty}\tau^n(X_\cdot')=\xi'\quad a.s..
 \end{align}
 Precisely $\bar\xi\leq \xi'$ by \eqref{taun}. On the other hand, on the set where $\bar \xi<\xi'$, we have $z'_{\bar\xi}\in Q$ since $\bar\xi<\xi'$, we also have $z'_{\bar\xi}=\partial$ since $z'_{\bar\xi}$ is the limit of points getting outside of any $Q^n$.
 Observe that before $\tau^n(X_\cdot')$, $X_t'$ also satisfies the SDE \eqref{xn}, by the local strong uniqueness of equation \eqref{xn} proved by Lemma \ref{lemm3.1}, we get $X_t^n=X_t'$ for $t\leq \tau^n(X_\cdot')$, so $\tau^n(X_\cdot')=\tau_{n,n}$.  And by \eqref{barxi} we see that
 $$\xi'=\bar\xi=\lim_{n\rightarrow\infty}\tau^n(X_\cdot')=\lim_{n\rightarrow\infty}\tau_{n,n}=\xi\quad a.s.,$$
 which implies that for $t\leq \xi=\xi'$, and $z_t'$ coincides with $z_t$ from our above construction \eqref{cz}.
\end{proof}

\section{Preparations of the proof of Theorem \ref{th2.3}}
 \subsection{Probabilistic representation of solutions to parabolic partial differential equations}
 In this subsection, we give a probabilistic representation of the solution to the following backward parabolic partial differential equation with a potential term $V(t,x):[0,T]\times\mathbb{R}^d\rightarrow\mathbb{R}$,
 \begin{equation}\label{kpde}
\left\{
\begin{aligned}
D_tu(t,x)+\mathcal{L}u(t,x)+V(t,x)u(t,x)&=0,\quad 0\leq t\leq T, \\
u(T,x)&=f(x). \\
\end{aligned}
\right.
\end{equation}
Here $T\in(0,\infty)$ and
\begin{align*}
\mathcal{L}u(t,x):=\frac{1}{2}\sum_{i,j=1}^da_{ij}(t,x)\frac{\partial^2u}{\partial x_i\partial x_j}(t,x)+b(t,x)\cdot \nabla u(t,x),\quad u\in\mathcal{C}_c^\infty(\mathbb{R}^{d+1}),
\end{align*}
 where $(a_{ij})_{1\leq i,j\leq d}=\sigma\sigma^*$. We first give the assumptions which make the representation formula hold.
\begin{assumption}\label{assfey}
(i) For all $1\leq i,j\leq d$, $\sigma_{ij}\in \mathcal{C}([0,T]\times\mathbb{R}^d)$,\\
%(ii) for all $t\in[0,T]$, $b(t,\cdot)$ and $\sigma(t,\cdot)\in\mathcal{C}(\mathbb{R}^d)$,\\
(ii) There exist positive constants $K$ and $\delta$ such that for all $(t,x)\in [0,T]\times{\mathbb{R}^d}$,
$$\delta|\lambda|^2\leq|\sigma^*(t,x)\lambda|^2\leq K|\lambda|^2,\quad \forall \lambda\in\mathbb{R}^d,$$
(iii) $b$, $V\in\mathcal{C}_b([0,T]\times\mathbb{R}^d)$,\\
(iv) For all $(t,x)$, $(s,y)\in [0,T]\times\mathbb{R}^d$, there exists constants $C_1$, $C_2$ and $C_3$ such that
\begin{align*}
|a_{ij}(t,x)-a_{ij}(s,y)|&\leq C_1(|x-y|\vee|t-s|^{1/2}),\\
|b(t,x)-b(s,y)|&\leq C_2(|x-y|\vee|t-s|^{1/2}),\\
|V(t,x)-V(s,y)|&\leq C_3(|x-y|\vee|t-s|^{1/2}).
\end{align*}
(v) $f\in\mathcal{C}_c^2({\mathbb{R}^d})$.
\end{assumption}
\begin{theorem}\label{Feymann}
If Assumption \ref{assfey} holds, then there exists a unique solution $u(t,x)$ to the equation \eqref{kpde} and it can be represented by the following formula
\begin{align}\label{fkf}
u(t,x)=E\Big[f(X(T,t,x))e^{\int_t^TV(u,X(u,t,x))du}\Big],\quad (t,x)\in[0,T]\times\mathbb{R}^d,
\end{align}
where $X(T,t,x)$ is the solution to the SDE \eqref{eq1} with initial point $(t,x)\in[0,T]\times\mathbb{R}^d$.
Furthermore, for $t\in[0,T)$ we have
\begin{align}{\label{L1}}
u(t,\cdot),\quad D_tu(t,\cdot),\quad \nabla u(t,\cdot),\quad \nabla^2u(t,\cdot)\in L^1(\mathbb{R}^d).
\end{align}

\end{theorem}
\begin{proof}
On one hand by classical results of partial differential equation (see \cite[Theorem 5.1]{pdeb}), we know that under our assumption there exists a unique solution $u(t,x)\in \mathcal{C}^{1,2}([0,T],\mathbb{R}^d)$ to the equation \eqref{kpde}, which can be written in the form of a potential with kernel $k$  (see \cite[(14.2)]{pdeb}):
\begin{align*}
u(t,x)=\int_{\mathbb{R}^d}k(T,y;t,x)f(y)dy,\quad (t,x)\in[0,T]\times\mathbb{R}^d
\end{align*}
satisfying
\begin{align*}
\lim_{t\rightarrow T}u(t,x)=\lim_{t\rightarrow T}\int_{\mathbb{R}^d}k(T,y;t,x)f(y)dy=f(x),
\end{align*}
and for $s=0,1,2$ there exists a constant $C$ such that for $0\leq t<T$ (see \cite[(13.1)]{pdeb})
\begin{align*}
\partial_x^sk(T,y;t,x)\leq C(T-t)^{-\frac{d+s}{2}}\exp\Big(-C\frac{|y-x|^2}{T-t}\Big).
\end{align*}
Then for $s=0,1,2$, for $t\in[0,T)$ we have
\begin{align*}
\int_{\mathbb{R}^d}|\partial_x^su(t,x)|dx\leq&\int_{\mathbb{R}^d}\int_{\mathbb{R}^d}|f(y)\partial_x^sk(T,y;t,x)|dydx
\\ \overset{Fubini}=&\int_{\mathbb{R}^d}|f(y)|\int_{\mathbb{R}^d}|\partial_x^sk(T,y;t,x)|dxdy
\\ \leq &C(T-t)^{-\frac{s}{2}}\int_{\mathbb{R}^d}|f(y)|dy<\infty,
\end{align*}
which implies that for $t\in[0,T),$
\begin{align}\label{leftpart}
u(t,\cdot),\quad \nabla u(t,\cdot),\quad \nabla^2u(t,\cdot)\in L^1(\mathbb{R}^d).
\end{align}
 Since $b$ is bounded, we get $D_tu(t,\cdot)\in L^1(\mathbb{R}^d)$ following from the equation \eqref{kpde} and \eqref{leftpart}.  On another hand, from our assumption we know that $\sigma\sigma^*$ is uniformly elliptic, $b(t,x)$ and  $\sigma_{ij}(t,x)$, $1\leq i,j\leq d$  are bounded for $(t,x)\in [0,T)\times\mathbb{R}^d$ and continuous in $t$ and Lipschitz continuous in $x$, by a known result (eg. see \cite[IV Theorem 2.2]{Ikeda}) we get the existence and uniqueness of the global solution $(X_t)_{t\geq0}$ to the SDE \eqref{eq1}. Then by \cite[Theorem 8.2.1]{BO} we get that \eqref{fkf} solves the equation \eqref{kpde}. Hence combining these two sides we get the desired result and also \eqref{L1} holds.
\end{proof}

\subsection{Some auxiliary proofs}
 In order to show that under certain conditions our solutions will not blow up, we need some auxiliary proofs which we collect in this subsection. We fix an $T\in (0,\infty)$, for $t\in[0,T]$ define
 \begin{align*}
 Q_{ T}:=(0,T)\times \mathbb{R}^d, \quad B_r:=\{x\in\mathbb{R}^d:|x|< r\},\quad Q^{t,r}:=[0,t)\times B_r.
 \end{align*}
 %Let $B$ be open bounded sets such that $Q^{1,1}\subset  B\subset Q_{T}$.
 %Let $\tau_E(X_\cdot):=\inf\left\{t\geq 0:(t,X_t)\notin E\right\}$ denotes the explosion time of process $(t,X_t)$ with respect to the domain $E\subset $.
\begin{assumption}\label{ass3}
(i) $\psi$ is a nonnegative function defined on $\mathbb{R}^{d+1}$  and $\psi\in\mathcal{C}_b^\infty(\mathbb{R}^{d+1})$,\\
(ii) $|\nabla\psi|\in\mathbb{L}_{p}^{q,loc}$ with $p,q\in(2,\infty)$ and $d/p+2/q<1$,\\
%(iii)$\psi$  blows up near the parabolic boundary of a open set $B\subset\mathbb{R}^{d+1}$,\\
%(iii) $D_t\psi(t,\cdot)\in\mathcal{C}_b^2(\mathbb{R}^d)$ for $t\in[0,T]$,\\
(iii) $\sigma$ satisfies the conditions in Assumption  \ref{ass3.2} (i),\\
%(iv) $\partial\sigma(t,\cdot)\in\mathcal{C}_b^2(\mathbb{R}^d)$ for $t\in[0,T]$ and $\partial^2\sigma$ is continuous on $[0,T]\times\mathbb{R}^d$,\\
(iv) For all $(t,x)$, $(s,y)\in [0,T]\times\mathbb{R}^d$, there exist constants $K_0$, $K\in[0,\infty)$ such that for all $1\leq i,j\leq d$,
\begin{align*}
|a_{ij}(t,x)-a_{ij}(s,y)|&\leq K(|x-y|\vee|t-s|^{1/2}),\\
|\partial_ja_{ij}(t,x)-\partial_ja_{ij}(s,y)|&\leq K_0(|x-y|\vee|t-s|^{1/2}).
\end{align*}
\end{assumption}
\noindent Let $(W_t)_{ t\geq0}$ be a $d-$dimensional Wiener process on  a given complete probability space $(\Omega,\mathcal{F},$ $(\mathcal{F}_t)_{t\geq0},P)$, denote $(a_{ij})_{1\leq i,j\leq d}=\sigma\sigma^*$. Let $(s,x)\in[0,\infty)\times\mathbb{R}^d$, we introduce the process $(Y(t,s,x))_{t\geq s}$,  satisfying
  \begin{align}\label{Y}
  Y(t,s,x)=x+\int_s^t\sigma(r,Y(r,s,x))dW_r+(\frac{1}{2}\sum_{j=1}^d\int_s^t\partial_ja_{ij}(r,Y(r,s,x))dr)_{1\leq i\leq d},
  \end{align}
  and process $(X(t,s,x))_{t\geq s}$ satisfying
  \begin{align}\label{X}
  X(t,s,x)=x+\int_s^t\sigma(r,X(r,s,x))dW_r+&(\frac{1}{2}\sum_{j=1}^d\int_s^t\partial_ja_{ij}(r,X(r,s,x))dr)_{1\leq i\leq d}
  \nonumber\\-&\int_s^t(\sigma\sigma^*\nabla\psi)(r,X(r,s,x))dr.
  \end{align}
  Since for $1\leq i,j\leq d$, $\partial_ja_{ij}=\sum_{k=1}^d\sigma_{ik}(\partial_j\sigma_{jk})+\sum_{k=1}^d(\partial_j\sigma_{ik})\sigma_{jk}$, and $|\nabla\sigma|\in\mathbb{L}_p^{q,loc}$, from Assumption \ref{ass3} (iii), we get $\sum_{j=1}^d|\partial_ja_{ij}|\in\mathbb{L}_p^{q,loc}$. Then Lemma \ref{lemm3.1} can be applied here to guarantee the existence and uniqueness of global $(\mathcal{F}_t)$-adapted solutions $(Y(t,s,x))_{t\geq s}$ and $(X(t,s,x))_{t\geq s}$ corresponding to SDEs  \eqref{Y} and \eqref{X}  if Assumption \ref{ass3} holds.   \\

\begin{lemma}\label{cor3.8}
Let Assumption \ref{ass3} be satisfied. Take a nonnegative Borel function $f$ on $\mathbb{R}^{d+1}$. For $t\in[0,T]$ introduce
\begin{align*}\beta_T(t,x)=\exp(&-\int_t^{T}\nabla\psi^*\sigma(s,Y(s,t,x))dW_s
-\frac{1}{2}\int_t^{T}|\nabla\psi^*\sigma\sigma^*\nabla\psi|(s,Y(s,t,x))ds
\\&-2\int_t^{T}D_t\psi(s,Y(s,t,x)ds),
%-\frac{1}{2}\sum_{i,j=1}^d\int_t^{T}\partial_ja_{ij}\partial_i\psi(s,Y(s,t,x))ds
\end{align*}
\begin{align*}
v_T(t,x)=E\beta_T(t,x)f(T,Y(T,t,x)),\quad c(t)=\int_{\mathbb{R}^d}e^{-2\psi(t,x)}v_T(t,x)dx.
\end{align*}
Then $c(t)$ is a constant for $t\in[0,T]$.
%If ${\color{red}\frac{1}{2}\sum_{i,j=1}^d\partial_ja_{ij}\partial_i\psi\geq 0,}$

\end{lemma}

\begin{proof}
 Using a standard approximation argument it suffices to prove the result for $f\in\mathcal{C}_c^\infty(\mathbb{R}^{d+1})$.
First  by Assumption \ref{ass3} (i) and (iii), we have
\begin{align*}
E\exp(\frac{1}{2}\int_t^{T}|\nabla\psi^*\sigma\sigma^*\nabla\psi|(s,Y(s,t,x))ds)<\infty.
\end{align*}
 Girsanov transformation yields
$$v_T(t,x)=E\exp(-\int_t^T2D_t\psi(s,X(s,t,x))ds)f(T,X(T,t,x)).$$

By Assumption \ref{ass3} (i), (iii) and (iv), we get that $(\frac{1}{2}\sum_{j=1}^d\partial_ja_{ij})_{1\leq i\leq d}$ and $-\sigma\sigma^*\nabla\psi$ are bounded and also satisfy  Assumption \ref{assfey}  (iv). By Theorem \ref{Feymann}, $v_T(t,x)$ is the solution to the following Kolmogrov equation with a potential term $-2D_t\psi$:
 \begin{equation}\label{kolmogrov}
\left\{
\begin{aligned}
D_tv_T(t,x)+\frac{1}{2}\sum_{i,j=1}^d\partial_j(a_{ij}\partial_iv_T(t,x))&-((\sigma\sigma^*\nabla\psi)^*\nabla v_T)(t,x)\\&-v_T(t,x)2D_t\psi(t,x)=0,\quad (t,x)\in[0,T]\times\mathbb{R}^d, \\
v_T(T,x) &= f(T,x). \\
\end{aligned}
\right.
\end{equation}
And Theorem \ref{Feymann} shows that for $t\in[0,T)$, $v_T(t,\cdot)$, $D_tv_T(t,\cdot)$, $\nabla v_T(t,\cdot)$, $\nabla^2v_T(t,\cdot)\in L^1(\mathbb{R}^d)$, also there exists a kernel $k(T,y;t,x)$ such that
$$v_T(t,x)=\int_{\mathbb{R}^d}k(T,y;t,x)f(T,y)dy$$
and there exists a constant $C$ such that (\cite[(13.1)]{pdeb})
\begin{align*}
D_tk(T,y;t,x)\leq C(T-t)^{-\frac{d+2}{2}}\exp\Big(-C\frac{|y-x|^2}{T-t}\Big).
\end{align*}\iffalse
Furthermore, from Assumption \ref{ass3} we know that all the  of coefficients of equation \eqref{kolmogrov} have at least $1-$ order H\"older continuity in $x$ and $\frac{1}{2}-$order H\"older continuity in $t$, then  by \cite[(13.3)]{pdeb}
\fi
Then by mean value theorem  for $h\in\mathbb{R}$ with $t+h\in(0,T)$ there exists an $\theta\in(0,1)$ such that
$$\frac{|k(T,y;t+h,x)-k(T,y;t,x)|}{h}= D_tk(T,y;t+\theta h,x)\leq C(T-t-\theta h)^{-\frac{d+2}{2}}\exp\Big(-C\frac{|y-x|^2}{T-t-\theta h}\Big),$$
then
\begin{align}\label{v}
\Big|\frac{v_T(t+h,x)-v_T(t,x)}{h}\Big|&\leq \int_{\mathbb{R}^d}\Big|\frac{k(T,y;t+h,x)-k(T,y;t,x)}{h}\Big|f(T,y)dy
\nonumber\\&\leq C(T-t-\theta h)^{-\frac{d+2}{2}}\int_{\mathbb{R}^d}\exp\Big(-C\frac{|y-x|^2}{T-t-\theta h}\Big)f(T,y)dy
\nonumber\\&\leq C'(T-t)^{-\frac{d+2}{2}}\int_{\mathbb{R}^d}\exp\Big(-C\frac{|y-x|^2}{T-t}\Big)f(T,y)dy
\end{align}
for  $|h|\ll1$.
Denote $g(t,x)=e^{-2\psi(t,x)}v_T(t,x)$,
%by Leibniz integral rule (eg. see \cite[Theorem 16.8]{PB})
we have for $t\in[0,T)$, $t+h\in(0,T)$,
\begin{align*}
\Big|\frac{g(t+h,x)-g(t,x)}{h}\Big|&= \Big|\frac{e^{-\psi(t+h,x)}(v_T(t+h,x)-v_T(t,x))}{h}+\frac{v_T(t,x)(e^{-\psi(t+h,x)}-e^{-\psi(t,x)})}{h}\Big|\\
&\leq \Big|\frac{v_T(t+h,x)-v_T(t,x)}{h}\Big|+\Big|\frac{v_T(t,x)(e^{-\psi(t+h,x)}-e^{-\psi(t,x)})}{h}\Big|
\\&\leq C'(T-t)^{-\frac{d+2}{2}}\int_{\mathbb{R}^d}\exp\Big(-C\frac{|y-x|^2}{T-t}\Big)f(T,y)dy+C''v_T(t,x)
\\&=:G_T(t,x),
\end{align*}
the last inequality holds because of \eqref{v} and mean value theorem. Since for for $t\in[0,T)$, $v_T(t,\cdot)\in L^1(\mathbb{R}^d)$ and
\begin{align*}
\int_{\mathbb{R}^d}(T-t)^{-\frac{d+2}{2}}\int_{\mathbb{R}^d}\exp\Big(-C\frac{|y-x|^2}{T-t}\Big)f(T,y)dydx\leq C(T-t)^{-1}\int_{\mathbb{R}^d}f(T,y)dy<\infty,
\end{align*}
it yields that $G_T(t,\cdot)\in L^1(\mathbb{R}^d)$. Then by dominated convergence theorem, we have
$$\lim_{h\rightarrow 0}\frac{\int_{\mathbb{R}^d}g(t+h,x)-g(t,x)dx}{h}=\lim_{h\rightarrow 0}\int_{\mathbb{R}^d}\frac{g(t+h,x)-g(t,x)}{h}dx=\int_{\mathbb{R}^d}D_tg(t,x)dx.$$
That is to say
\begin{align}\label{exchanget}
D_t\int_{\mathbb{R}^d}e^{-2\psi(t,x)}v_T(t,x)dx=\int_{\mathbb{R}^d}D_t(e^{-2\psi}v_T)(t,x)dx.
\end{align}
Besides, we can write the first equation in \eqref{kolmogrov} in an equivalent form as
\begin{align}\label{re}
D_t(e^{-2\psi}v_T)+\frac{1}{2}\sum_{i,j=1}^d\partial_i(e^{-2\psi}a_{ij}\partial_jv_T)=0.
\end{align}

Now we are going to prove
\begin{align*}
\int_{\mathbb{R}^d} div(F)(t,x)dx:=\int_{\mathbb{R}^d}\sum_{i,j=1}^d\partial_i(e^{-2\psi}a_{ij}\partial_jv_T)(t,x)dx=0, \quad t\in[0,T).
\end{align*}
Since $\psi$ is nonnegative, $\partial_i \psi$ and $a_{ij}$ are bounded on $[0,\infty)\times\mathbb{R}^d$ for $1\leq i,j\leq d$, then there exist constants $C_1$ and $C_2$ such that
 \begin{align*}
 F_i=\sum_{j=1}^de^{-2\psi}a_{ij}\partial_jv_T\leq C_1\sum_{j=1}^d|\partial_jv_T|,
  \end{align*}
  and
  \begin{align*}
 div (F)&=\sum_{i,j=1}^d\partial_i(e^{-2\psi}a_{ij}\partial_jv_T)\\
 &=\sum_{i,j=1}^d(-2\partial_i\psi e^{-2\psi}a_{ij}\partial_jv_T
 +\partial_i a_{ij}e^{-2\psi}\partial_jv_T +e^{-2\psi}a_{ij}\partial_i\partial_jv_T)
 \\ &\leq C_2 \sum_{i,j=1}^d(|\partial_j v_T|+|\partial_{i}\partial_{j}v_T|).
\end{align*}
According to \eqref{L1} we know that $F(t,\cdot),div F(t,\cdot)\in L^1(\mathbb{R}^d)$ for any $t\in[0,T)$.
For $n\in\mathbb{N}$, take smooth function $\chi_n$ on $\mathbb{R}^d$ such that $\chi_n(x)=1$ when $|x|\leq n$ and $\chi_n(x)=0$ when $|x|>n+2$. Then by dominated convergence theorem and integration by parts formula for $t\in[0,T)$,
\begin{align*}
\int_{\mathbb{R}^d} div(F)(t,x)dx=\lim_{n\rightarrow\infty}\int_{\mathbb{R}^d} \chi_n(x)div(F)(t,x)dx=-\lim_{n\rightarrow\infty}\int_{\mathbb{R}^d}\nabla \chi_n(x)\cdot F(t,x)dx=0.
\end{align*}
Hence from \eqref{re} and \eqref{exchanget} we get
\begin{align*}
D_t\int_{\mathbb{R}^d}e^{-2\psi(t,x)}v_T(t,x)dx=0.
\end{align*}
This yields that $c(t)$ is a constant for $t\in[0,T)$. Since $c(t)$ is continuous for $t\in[0,T]$, it shows that $c(t)$ is a constant for $t\in[0,T]$.
\end{proof}
Theorem \ref{Feymann} talks about Cauchy problem with terminal data for the equation \eqref{kpde} in the domain $[0,T]\times \mathbb{R}^d$. In the cylindrical domain $Q^{r^2,r}$ with surface $\partial Q^{r^2,r}:=((0,r^2)\times \partial B_r)\cup(\left\{r^2\right\}\times B_r)$ for $r\in(0,1]$, we consider the first boundary problem to the following parabolic equation on $ \overline {Q^{r^2,r}}$ with assuming that $f$ is a continuous function on $\partial Q^{r^2,r}$:
 \begin{equation}\label{QP}
\left\{
\begin{aligned}
\mathcal{L}u(t,x) &=D_tu(t,x)+\frac{1}{2}\sum_{i,j=1}^d\partial_i(a_{ij}(t,x)\partial_ju(t,x))=0 \quad\quad on \quad Q^{r^2,r}, \\
u(t,x) &= f(t,x) \quad\quad on \quad\partial Q^{r^2,r}, \\
\end{aligned}
\right.
\end{equation}
 where $(a_{ij})_{1\leq i,j\leq d}=\sigma\sigma^*$. If Assumption \ref{ass3} (iii) and (iv) hold,
 from \cite[Theorem 3.1]{LOTFI RIAHI} and \cite[Corollary 3.2]{LOTFI RIAHI} the solution $u(t,x)$ to  \eqref{QP} has a representation as following:
$$u(t,x)=\int_{\partial Q^{r^2,r}} f(s,y)p(s,y;t,x)dS(s,y),$$
where $dS$ denotes the surface measure on $\partial Q^{r^2,r}$, and $p(s,y;t,x)$ is the Poisson kernel on $Q^{r^2,r}$ corresponding to  \eqref{QP}, which has the following upper bound estimate on $Q^{r^2,r}$ (\cite{LOTFI RIAHI}) with a constant $c$ independent of $f$
\begin{align}\label{addlemm3.8}
p(s,y;t,x)\leq c(s-t)^{-\frac{(d+1)}{2}}\exp(-c\frac{|y-x|^2}{s-t})
\end{align}
for all $(t,x)\in Q^{r^2,r}$, $(s,y)\in \partial Q^{r^2,r}$, $0\leq t<s$.\\

\indent  Besides, we also can represent the solution to the above equation \eqref{QP} in a probabilistic way. For $(t,x)\in Q^{r^2,r}$, let $$\tau_r:=\inf\left\{s\geq0:(s,Y(s,t,x))\notin Q^{r^2,r}\right\},$$ by applying It\^o's formula to $u(s,Y(s,t,x))$ and taking expectation, we have for $(t,x)\in Q^{r^2,r}$,
$$ u(t,x)=E^{(t,x)}[u(\tau_r,Y({\tau_r},t,x))]-E^{(t,x)}[\int_t^{\tau_r}\mathcal{L}u(s,Y(s,t,x))ds]=E^{(t,x)}[f(\tau_r,Y({\tau_r},t,x))].$$
Hence
$$E^{(t,x)}[f(\tau_r,Y({\tau_r},t,x))]=\int_{\partial Q^{r^2,r}} f(s,y)p(s,y;t,x)dS(s,y).$$
We take $(0,0)$ as the start point of the process $(s,Y(s,t,x))$, then denote $Y_s:=Y(s,0,0)$ and $E[f(\tau_r,Y_{\tau_r})]=\int_{\partial Q^{r^2,r}} f(s,y)p(s,y;0,0)dS(s,y).$

\begin{lemma}\label{lemm3.9}
If Assumption \ref{ass3} (iii) and (iv) hold, then on an extension of the probability space there is a stopping time $\gamma$ such that the distribution of $(\gamma,Y_\gamma)$ has a bounded density concentrated on $Q^{1,1}$.
\end{lemma}

\begin{proof}
Let $n=d+3$. On an extension of our probability space there exists a random variable $\rho$ with values in $[0,1]$ and density function $h(r)=nr^{n-1}$ such that $\rho$ is independent of all $({\mathcal{F}}_t)_{t\geq0}$. Then $\rho$ is also independent to $(t,Y_t)_{t\geq0}$, since $(Y_t)_{t\geq0}$ is adapted to $(\mathcal{F}_t)_{t\geq0}$. Let $\hat{\mathcal{F}}_t=\mathcal{F}_t\vee\sigma(\rho),t\geq0$, and define $\gamma$ as the first exit time of $(t,Y_t)_{t\geq0}$ from $Q^{\rho^2,\rho}$. Then $\gamma$ is a bounded $(\hat{\mathcal{F}}_t)_{t\geq0}$ stopping time.  We claim that $\gamma$ is a random variable of the type that we are looking for.\\
 \indent Actually, according to independence and the above potential knowledge, for a nonnegative continuous function $f(t,x)$ on $[0,\infty)\times \mathbb{R}^d$ we have
  \begin{align*} Ef(\gamma, Y_\gamma)=&E[Ef(\tau_r,Y_{\tau_r})\bigg|_{\rho=r}]=E[\int_{\partial Q^{r^2,r}}f(s,y)p(s,y;0,0)dS(s,y)\bigg|_{\rho=r}]\\=&\int_0^1h(r)dr\int_{\partial Q^{r^2,r}}f(s,y)p(s,y;0,0)dS(s,y)\\=&
  \int_0^1 h(r)dr\int_{(0,r^2)\times \partial B_r}f(s,y)p(s,y;0,0)dS(s,y)\\&+\int_0^1h(r)dr\int_{B_r}f(r^2,y)p(r^2,y;0,0)dy=:I_1+I_2.\end{align*}
  Then \eqref{addlemm3.8} and the fact that $\exp(-c\frac{|y|^2}{s})s^{-(d+1)/2}$ is bounded by $Nr^{-d-1}$ on $(0,r^2)\times\partial B_r$ yield
  \begin{align*} I_1&\leq k\int_0^1h(r)dr\int_0^{r^2}\int_{\partial B_r}f(s,y)\frac{\exp(-c\frac{|y|^2}{s})}{s^{(d+1)/2}}dS(s,y)
  \\&\leq N\int_0^1h(r)r^{-d-1}dr\int_0^{r^2}\int_{\partial B_r}f(s,y)dS(s,y)
  \\&
  \leq N\int_0^1\int_0^{r^2}\int_{\partial B_1}r^{-d-1}f(s,ry)h(r)r^{d-1}d(\partial B_1)dsdr\\&
  \leq N\int_0^1\int_0^1\int_{\partial B_1}f(s,ry)r^{d}d(\partial B_1)dsdr\leq N\int_{Q^{1,1}}f(s,y)dsdy,\end{align*}
and
  \begin{align*} I_2&\leq k\int_0^1\int_{B_r}f(r^2,y)h(r)\frac{\exp(-c\frac{|y|^2}{r})}{r^{d+1}}dydr\\&\leq N\int_0^1\int_{B_r}f(r^2,y)h(r)r^{-d-1}dydr\\&
  =N\int_0^1\int_{B_r}f(r^2,y)r^{n-2-d}dydr\leq N\int_{Q^{1,1}}f(s,y)dsdy.\end{align*}
  Hence
  $$Ef(\gamma,Y_\gamma)\leq N\int_{Q^{1,1}}f(t,x)dxdt$$
and N is independent of $f$. \\
\indent For arbitrary nonnegative function $|fI_{Q^{1,1}}|\in \mathbb{L}_1^1$, we can use a standard method to approximate $f$ via continuous functions. The conclusion is proved.
\end{proof}

\begin{lemma}\label{lemm3.10}
Let Assumption \ref{ass3} hold. Let $K_2\in[0,\infty)$ be a constant. Assume that for some $p$, $q$ satisfying $p,q\in(2,\infty)$ and $ \frac{d}{p}+\frac{2}{q}<1,$ we have
$$ \quad\psi I_{Q^{1,1}}\leq K_2, \quad \Vert\nabla\psi I_{Q^{1,1}}\Vert_{\mathbb{L}_p^q}\leq K_2.\quad$$
%{\color{red} \frac{1}{2}\sum_{i,j=1}^d\partial_ja_{ij}\partial_i\psi\geq 0,}
 Take an $r\in(1,\infty)$ and a nonnegative Borel function $f=f(t,x)$ on $(0,\infty)\times\mathbb{R}^d$ such that $f(t,x)=0$ for $t>T$. For $0\leq s\leq t\leq T$ and $x\in\mathbb{R}^d$ introduce
\begin{align*}
\rho_t(s,x)&=\exp(-\int_s^{t}\nabla\psi^*\sigma(u,Y(u,s,x))dW_s-\frac{1}{2}\int_s^{t}|\nabla\psi^*\sigma\sigma^*\nabla\psi|(u,Y(u,s,x))du),\\
\alpha_t(s,x)&=\exp(-2\int_s^{t}(D_t\psi)_+(u,Y(u,s,x))du),\\
%+1/4\sum_{i,j=1}^d\partial_i\psi\partial_ja_{ij}
u_t(s,x)&=E\rho_t(s,x)\alpha_t(s,x)f(t,Y(t,s,x)).
\end{align*}
Then there is a constant $N$, depending only on  $K$, $r$, $p$, $q$, $K_2$ and $T$, such that
\begin{equation}\label{eq3.8}
\int_0^Tu_t(0,0)dt\leq N(\int_{(0,\infty)\times\mathbb{R}^d}f^re^{-2\psi}dtdx)^{1/r}+N(\int_{Q^{1,1}}f^{d+3}dtdx)^{1/(d+3)}.
\end{equation}
\end{lemma}

\begin{proof}
By the strong Markov property of $(Y_t)_{t\geq0}$, { which can be obtained  from the similar argument as in the proof of Theorem \ref{mainadd1}}, for any stopping time $\tau$ we have
$$EI_{\tau\leq t}\rho_t(0,0)\alpha_t(0,0)f(t,Y_t)=EI_{\tau\leq t}\rho_\tau(0,0)\alpha_\tau(0,0)u_t(\tau,Y_\tau).$$
Therefore, upon assuming without losing generality that $T\geq1$, for $\gamma$ from Lemma \ref{lemm3.9},
$$\int_0^Tu_t(0,0)dt=E\int_0^\gamma\rho_t(0,0)\alpha_t(0,0)f(t,Y_t)dt+E\rho_\gamma(0,0)\alpha_\gamma(0,0)\int_\gamma^Tu_t(\gamma,Y_\gamma)dt=:I_1+I_2.$$
Observe that $\alpha_t\leq1$ and for $t\leq\gamma$ we have $(t,Y_t)\in Q^{1,1}$ so that, in particular, in the formula defining $\rho_t(0,0)$ we can replace $\nabla\psi$ with $\nabla\psi I_{Q^{1,1}}$ and hence all moments of $\rho_t(0,0)I_{t\leq\gamma}$ and $\rho_\gamma(0,0)$ are finite and uniformly bounded in $t$. Since by  \eqref{eq3.42} we have
$$E[\exp(\frac{1}{2}\int_0^t|\nabla\psi^*\sigma\sigma^*\nabla\psi|I_{Q^{1,1}}(u,Y(u,s,x))du)]<\infty$$
for all $t\in[0,T]$. For the moments of $\rho_t(0,0)I_{t\leq\gamma}$ and $\rho_\gamma(0,0)$, by using the same way of getting \eqref{exponential} and \eqref{en} we get the desired results.  We also can replace $\frac{1}{2}\sum_{j=1}^d\int_s^{t}\partial_ja_{ij}(r,Y(r,s,x))dr$ by $\frac{1}{2}\sum_{j=1}^d\int_s^{t}I_{Q^{1,1}}\partial_ja_{ij}(r,Y(r,s,x))dr)$ in the SDE \eqref{Y} for all $1\leq i,j\leq d$, it follows by H\"older's inequality and \eqref{expx} that for any $v\in(1,\infty)$
$$I_1\leq N(E\int_0^T|f^vI_{Q^{1,1}}|(t,Y_t)dt)^{1/v}\leq N\Vert f^vI_{Q^{1,1}}\Vert_{\mathbb{L}^{d+5/2}}^{1/v}.$$
We can choose $v$ so that $v(d+5/2)=d+3$, and get that $I_1$ is less than the second term on the right in \eqref{eq3.8}.\\
\indent In what concerns $I_2$ we again use $\alpha_\gamma(0,0)\leq1$ and the finiteness of all moments of $\rho_\gamma(0,0)$. Then we find
\begin{equation}\label{newlem3.10}
I_2\leq N(\int_0^1\int_s^T(\int_{B_1}u_t^r(s,x)dx)dtds)^{1/r}.
\end{equation}
To estimate the interior integral with respect to $x$ we insert there $\exp(-2\psi(s,x))$ and again use H\"older's inequality and the fact that $E\rho_t(s,x)\leq1$. This yields
$$I_2(s,t):=\int_{B_1}u_t^r(s,x)dx\leq e^{2K_2}\int_{\mathbb{R}^d}e^{-2\psi(s,x)}\hat{v}_t(s,x)dx$$
where
$$\hat{v}_t(s,x)=E\rho_t(s,x)\alpha_t(s,x)f^r(t,Y({t},s,x))\leq E\beta_t(s,x)f^r(t,Y({t},s,x)).$$
Hence by Lemma \ref{cor3.8},
$$I_2(s,t)\leq e^{2K_2}\int_{\mathbb{R}^d}e^{-2\psi(t,x)}f^r(t,x)dx,$$
 which shows that $I_2$ is less than the first term on the right in \eqref{eq3.8}. The Lemma is proved.
\end{proof}

\begin{lemma}\label{lemm3.11}
 Let the assumptions of Lemma \ref{lemm3.10} be satisfied and let $\epsilon\in[0,2)$ be a constant and $h$ a nonnegative Borel function on bounded domain $Q\subset[0,\infty)\times\mathbb{R}^{d}$ such that on $Q$,
\begin{equation}\label{eq3.9}
2D_t\psi+\sum_{i,j=1}^{d}\partial_j(a_{ij}\partial_i\psi)\leq he^{\epsilon\psi}.
\end{equation}
%\frac{1}{2}\sum_{i,j=1}^d\partial_ja_{ij}\partial_i\psi\geq 0,
Then for any $\delta\in[0,2-\epsilon)$, $r\in(1,2/(\delta+\epsilon)]$,  there exists a constant $N$, depending only on $K$, $T$, $p$, $q$, $K_2$, $\epsilon$, $\delta$ and r (but not $Q$) such that for any stopping time $\tau\leq\tau_Q(Y_\cdot)$ we have
\begin{equation}\label{eq3.10}
E\Phi_\tau\leq N +N(\int_Qh^re^{-(2-r\eta)\psi}dtdx)^{1/r}+N\sup_{Q^{1,1}}h,
\end{equation}
where $\eta=\delta+\epsilon$ so that $r\eta\leq 2$ and
\begin{align*}\Phi_t:=\exp(&-\int_0^t(\nabla\psi^*\sigma)(s,Y_s)dW_s-\frac{1}{2}\int_0^t|\nabla\psi^*\sigma\sigma^*\nabla\psi|(s,Y_s)ds
\\&-2\int_0^t(D_t\psi)_+(s,Y_s)ds+\delta\psi(t,Y_t)).\end{align*}
%+1/4\sum_{i,j=1}^d\partial_i\psi\partial_ja_{ij}
\end{lemma}

\begin{proof}
By It\^o's formula,
\begin{align*}
\Phi_\tau=&\Phi_0+m_\tau+\int_0^\tau\Phi_t[\delta D_t\psi+\frac{\delta}{2}\sum_{i,j=1}^{d}\partial_j(a_{ij}\partial_i\psi)-2(D_t\psi)_+
\\&\quad\quad\quad\quad\quad\quad\quad\quad+\frac{1}{2}(|\delta-1|^2-1)|\nabla\psi^*\sigma\sigma^*\nabla\psi|](t,Y_t)dt\\
%+1/4\sum_{i,j=1}^d\partial_i\psi\partial_ja_{ij}
%=&\Phi_0+m_\tau+\int_0^\tau\Phi_t[\delta D_t\psi+\frac{\delta}{2}\sum_{i,j=1}^{d}\partial_j(a_{ij}\partial_i\psi)-2(D_t\psi)_+\\&\quad\quad\quad\quad\quad\quad\quad\quad+\frac{1}{2}(|\delta-1|-1)|\partial\psi^*\sigma\sigma^*\partial\psi|^2](t,Y_t)dt,
\end{align*}
%+1/4\sum_{i,j=1}^d\partial_i\psi\partial_ja_{ij}
where $m_t$ is a local martingale starting at zero. By using \eqref{eq3.9}, and the inequality $|\delta-1|\leq1$ we obtain
\begin{equation}\label{eq3.11}\Phi_\tau\leq \Phi_0+\delta\int_0^\tau\Phi_th(t,Y_t)\exp(\epsilon\psi(t,Y_t))dt+m_\tau.
\end{equation}
Since $\Phi_t\geq0$ we take the expectations of both sides and drop $Em_\tau$. More precisely, we introduce $\tau_n:=\inf\left\{t\geq0:|m_t|\geq n\right\}$ and substitute $\tau\wedge\tau_n$ in place of $\tau$ in \eqref{eq3.11}. After that we take expectations, use the fact that $Em_{\tau\wedge\tau_{n}}=0$, let $n\rightarrow\infty,$ and finally use Fatou's Lemma with monotone convergence theorem. Furthermore, we denote $f=I_Qh\exp(\eta\psi)$ and notice that $\tau\leq T$. Then in the notation of Lemma \ref{lemm3.10}, we find that
\begin{align*} E\Phi_\tau&\leq N+NE\int_0^\tau\rho_t(0,0)\alpha_t(0,0)f(t,Y_t)dt\\&\leq N+N\int_0^TE\rho_t(0,0)\alpha_t(0,0)f(t,Y_t)dt=N+N\int_0^Tu_t(0,0)dt.\end{align*}
It only remains to note that the first term in the right-hand side of \eqref{eq3.8} is just the second one on the right in \eqref{eq3.10} and the second integral on the right in \eqref{eq3.8} is less than $volQ^{1,1}\sup_{Q^{1,1}}h^{d+3}\exp[\eta K_2(d+3)].$ The Lemma is proved.
\end{proof}

\begin{theorem}\label{th3.12}
Let Assumption \ref{ass3} hold. Let $K_1$, $K_2\in[0,\infty)$ and $\epsilon\in[0,2)$ be some constants and let $Q$ be a bounded subdomain of $Q_{T}$ and $h$ be a nonnegative Borel function on $Q$. Assume that for some $p$, $q$ satisfying $p,q\in(2,\infty)$ and $\frac{d}{p}+\frac{2}{q}<1,$ we have
$$hI_{Q^{1,1}}\leq K_2,\quad \psi I_{Q^{1,1}}\leq K_2,\quad \Vert I_{Q^{1,1}\nabla\psi }\Vert_{\mathbb{L}_p^q}\leq K_2.$$
Also assume that on $Q$
\begin{align*}
\psi\geq0,\quad 2D_t\psi\leq K_1\psi,\quad \\ 2D_t\psi+\sum_{i,j=1}^{d}\partial_j(a_{ij}\partial_i\psi)\leq he^{\epsilon\psi}.
\end{align*}
%0\leq\sum_{i,j=1}^d\nu x_i\partial_ja_{ij}\leq \nu|x|^2
%+\frac{1}{2}\sum_{i,j=1}^d\partial_i\psi\partial_ja_{ij}, \frac{1}{2}\sum_{i,j=1}^d\partial_ja_{ij}\partial_i\psi\geq 0,
 Denote by $X_t$, $t\in[0,T]$, the solution of
 $$X_t=\int_0^t\sigma(s,X_s)dW_s+\int_0^t(-\sigma\sigma^*\nabla\psi)(s,X_s)ds+(\frac{1}{2}\sum_{j=1}^d\int_0^t\partial_ja_{ij}(s,X_s)ds)_{1\leq i\leq d}.$$
 Then for any $r\in(1,4/(2+\epsilon)]$ there exists a constants $N$, depending only on  $K$, $K_1$, $K_2$, $r$, $d$, $T$, $p$, $q$,  and $\epsilon$, such that
\begin{equation}\label{eq3.12}
E\sup_{t\leq\tau_Q(X_\cdot)}\exp[\mu(\psi(t,X_t)+\nu|X_t|^2)]\leq N+NH_Q(T,a,r)
\end{equation}
  where $H_Q$ is introduced in Assumption \ref{ass2.1}, $a=(2-r\eta)\nu$, $\eta=2\delta+\epsilon$, $\mu$, $\nu$ and $\delta$ are taken from \eqref{eq2.5}. Here $\tau_Q(X_{\cdot}):=\inf\left\{t\geq 0:(t,X_t)\notin Q\right\}$.
  \end{theorem}

\begin{proof}
Define $\hat\psi=\psi+\nu|x|^2$,
$$M_t=\exp(\delta\hat\psi(t,X_t)-\frac{K_1}{2}\int_0^t\hat\psi(s,X_s)ds),\quad M_*=\sup_{t\leq\tau_Q(X_\cdot)}M_t.$$
Then for $t\leq\tau_Q(X_\cdot)$,
$$\hat\psi(t,X_t)\leq \ln M_*^{1/\delta}+\frac{K_1}{2\delta}\int_0^t\hat\psi(s,X_s)ds$$
and hence by Gronwall's inequality
$$\hat\psi(t,X_t)\leq e^{tK_1/(2\delta)}\ln M_*^{1/\delta}\leq e^{TK_1/(2\delta)}\ln M_*^{1/\delta}.$$
Take $\mu=\frac{\delta}{2}e^{-TK_1/(2\delta)}$, then
\begin{equation}\label{eqn3.12}
\exp(\mu\hat\psi(t,X_t))\leq  \sqrt{M_*}.
\end{equation}
Therefore, to prove \eqref{eq3.12}, it suffices to prove that $E\sqrt{M_*}\leq N$. It turns by a well known result on transformations of stochastic inequalities (see Lemma 3.2 in \cite{I.N.}), if $EM_\tau\leq N_1$ for all stopping times $\tau\leq\tau_Q(X_\cdot)$. Then $E\sqrt{M_*}\leq 3N_1$. Thus, it suffices to estimate $EM_\tau$.\\
\indent On a probability space carrying a $d-$dimensional Wiener process $(\hat {W}_t)_{t\geq0}$ introduce $(\hat {X}_t)_{t\geq0}$ as the solution of the equation
\begin{equation}\label{eq3.13}
\hat{X}_t=\int_0^t\sigma(s,\hat{X}_s)d\hat{W}_s-\int_0^{t\wedge \tau_Q(\hat{X_\cdot})}\sigma\sigma^*\nabla\hat\psi(s,\hat X_s)ds+(\int_0^{t\wedge \tau_Q(\hat{X_\cdot})}\frac{1}{2}\sum_{j=1}^d\partial_j a_{ij})(s,\hat X_s)ds)_{1\leq i\leq d}.
\end{equation}
Also set
$$\hat {M}_t=\exp(2\delta\hat\psi(t,\hat{X}_t)-2\int_0^t(D_t\hat\psi)_+(s,\hat{X}_s)ds),\quad t\geq0.$$
%+1/4\sum_{i,j=1}^{d}\partial_ja_{ij}\hat\partial_i\psi
Write $\hat E$ for the expectation sign on the new probability space and observe that on $Q$
\begin{align}\label{conass}
2D_t\hat\psi+\sum_{i,j=1}^{d}\partial_j(a_{ij}\partial_i\hat\psi)&=2D_t\psi+\sum_{i,j=1}^{d}\partial_j(a_{ij}\partial_i\psi)+2\nu\sum_{i,j=1}^dx_i\partial_j a_{ij}+2\nu\sum_{i,j=1}^d\partial_ja_{ij}\nonumber\\&\leq (h+C)e^{\epsilon\hat\psi}.
\end{align}
Here $2\nu\sum_{i,j=1}^dx_i\partial_j a_{ij}+2\nu\sum_{i,j=1}^d\partial_ja_{ij}\leq (h+C)e^{\epsilon\hat\psi}$ holds because of Assumption \ref{ass3}, which means $|\partial_ja_{ij}|$ is bounded. Then after an obvious change of measure (cf. Lemma \ref{Girsanov} ) inequality \eqref{eq3.10} with $2\delta$, $\hat E$, $\hat\psi$, and $\hat{W}_t$ in place of $\delta$, $E$, $\psi$, and $W_t$, respectively, $\eta=2\delta+\epsilon$, and $r\in(1,4/(2+\epsilon)]\subset(1,2/(2\delta+\epsilon)] $ is written as
$$\hat E\hat{M}_\tau\leq N+N(\int_Qh^rI_{(0,T)}e^{-(2-r\eta)\hat\psi}dtdx)^{1/r}$$
and since $\hat\psi\geq \nu|x|^2$ on $Q$, we obtain
$$\hat E\hat {M}_\tau\leq N+N(\int_Qh^rI_{(0,T)}e^{-(2-r\eta)\nu|x|^2}dtdx)^{1/r}= N+N H_Q^{1/r}(T,(2-r\eta)\nu,r)=:N_0$$
for all stopping times $\tau\leq\tau_Q(\hat{X}_\cdot)$, which yields
$$\hat E\sqrt{\hat M_*}\leq 3N_0.$$
Combining this with the inequality
$$\exp(2\delta\hat\psi(t,\hat{X}_t)-{K_1}\int_0^t\hat\psi(x,\hat{X}_s)ds)\leq {\hat{M}_t},\quad\quad t\leq\tau_{Q}(\hat{X}_\cdot),$$
the left-hand side of which is quite similar to $M_t$ but with $2\hat\psi$ in place of $\hat\psi$, the above argument deduce
\begin{equation}\label{eq3.14}
\hat E\sup_{t\leq\tau_Q(\hat{X}_\cdot)}\exp(2\mu \nu|\hat{X}_t|^2)\leq \hat E\sup_{t\leq\tau_Q(\hat{X}_\cdot)}\exp(2\mu\hat\psi(t,\hat{X}_t))\leq
NN_0.
\end{equation}
We now estimate $EM_\tau$ through $\hat E\hat {M}_\tau$ by using Girsanov's theorem and H\"older's inequality. We use a certain freedom in choosing $\hat{X}_t$ and $\hat{W}_t$ and on the probability space where $W_t$ and $X_t$ are given we introduce a new measure by the formula:
$$\hat P(d\omega)=\exp(-2\nu\int_0^\infty X_t^* \sigma(t,X_t)I_{t<\tau_Q(X_\cdot)}dW_t-2\nu^2\int_0^\infty X _t^*(\sigma\sigma^{*})(t,X_t)X_tI_{t<\tau_Q(X_\cdot)}dt)P(d\omega).$$
Since $Q$ is a bounded domain, then we have
 $$E\exp\Big(2\nu^2\int_0^\infty X _t^*(\sigma\sigma^{*})(t,X_t)X_tI_{t<\tau_Q(X_\cdot)}dt\Big)\leq E\exp\Big(2\nu^2K\int_0^T X _t^*X_tI_{t<\tau_Q(X_\cdot)}dt\Big)<\infty,$$
 which implies that $\hat P$ is a probability measure. Furthermore, as is easy to see, for $t\leq\tau_Q(X_\cdot)$
$$\hat{X}_t:=X_tI_{t<\tau_Q(X_\cdot)}+(\int_0^t\sigma(s,X_s)dW_s-\int_0^{\tau_Q(X_\cdot)}\sigma(s,X_s)dW_s+X_{\tau_Q(X_\cdot)})I_{t\geq\tau_Q(X_\cdot)}$$
coincides with $X_t$  and satisfies \eqref{eq3.13} for $t\leq\tau_Q(X_\cdot)$ with
$$\hat{W}_t=W_t+2\nu\int_0^{t\wedge\tau_Q(X_\cdot)}\sigma^*(s,X_s)X_sds$$
which is a Wiener process with respect to $\hat P$. In this situation for $\tau\leq\tau_Q(X_\cdot)=\tau_Q(\hat{X}_\cdot)$
\begin{align*} EM_\tau&\leq\hat E\hat{M}_\tau^{1/2}\exp(2\nu\int_0^\infty \hat X_t^*\sigma(t,\hat X_t)I_{t<\tau_Q(\hat X_\cdot)}d\hat W_t-2\nu^2\int_0^\infty \hat X_t^*(\sigma\sigma^*)(t,\hat X_t)\hat X_tI_{t<\tau_Q(\hat X_\cdot)}dt)\\&
\leq(\hat E\hat {M}_\tau)^{1/2}(\hat E\rho^{1/2}\exp(12v^2\int_0^\infty \hat X_t^*(\sigma\sigma^*)(t,\hat X_t)\hat X_tI_{t<\tau_Q(\hat X_\cdot)}dt))^{1/2}\end{align*}
where
$$\rho=\exp(8\nu\int_0^\infty \hat X_t^*\sigma(t,\hat X_t)I_{t<\tau_Q(\hat X_\cdot)}d\hat W_t-32\nu^2\int_0^\infty \hat X_t^*(\sigma\sigma^*)(t,\hat X_t)\hat X_tI_{t<\tau_Q(\hat X_\cdot)}dt).$$
Observe that $\hat E\rho=1$ and $\hat E\hat {M}_\tau\leq N_0$. Therefore,
$$EM_\tau\leq N_0^{1/2}(\hat E\exp(24\nu^2\int_0^{\tau_Q({X}_\cdot)}(\hat X_t^*(\sigma\sigma^*)(t,\hat X_t)\hat X_t)dt))^{1/4}.$$
It only remains to refer to \eqref{eq3.14} after noticing that
$$24\nu^2\int_0^{\tau_Q({X}_\cdot)}(\hat X_t^*(\sigma\sigma^*)(t,\hat X_t)\hat X_t)dt\leq 24\nu^2KT\sup_{t\leq\tau_Q({X}_\cdot)}|{X}_t|^2=2\mu \nu \sup_{t\leq\tau_Q({X}_\cdot)}|{X}_t|^2$$
and use the inequality $\iota^\alpha\leq 1+\iota$ if $\iota\geq 0$, $0\leq \alpha\leq 1$, where $\nu=\mu/(12KT)$. The theorem is proved.
\end{proof}

\section{Proof of Theorem \ref{th2.3}}\label{sec4}
 By Theorem \ref{mainadd1}  the strong solution $(t,X_t)_{t\geq0}$ to \eqref{eq2.4} is defined at least until the time $\xi$ when $(s+t,X_t)_{t\geq0}$ exits from all $Q^n$. We claim that in order to prove $\xi=\infty $ $a.s.$ and also to prove the second assertions of the theorem, it suffices to prove that for each $T\in(0,\infty)$ and $m\geq 1$ there exists a constant $N$, depending only on $K$,  $K_1$, $d$, $p(m+1)$, $q(m+1)$, $\epsilon$, $T$, $\Vert I_{Q^{m+1}}\nabla\phi \Vert_{\mathbb{L}_{p(m+1)}^{q(m+1)}}$, $dist(\partial Q^m,\partial Q^{m+1})$, $\sup_{Q^{m+1}}\left\{\phi+h\right\}$, and the function $H$, such that for $(s,x)\in Q^m$ we have
\begin{equation}\label{eq4.1}
E\sup_{t<\xi\wedge T}\exp(\mu\phi(s+t,X_t)+\mu \nu|X_t|^2)\leq N.
\end{equation}
  \indent To prove the claim notice that \eqref{eq4.1} implies
  \begin{equation}\label{eq4.2}
  \sup_{t<\xi\wedge T}(\phi(s+t,X_t)+|X_t|^2)<\infty\quad a.s..
  \end{equation}
  It follows that $a.s.$ there exists an $n\geq1$ such that up to time $\xi\wedge T$ the trajectory $(Z_t)_{t\geq0}=(s+t, X_t)_{t\geq0}$ lies in $Q^n$. Indeed, on the set of all $\omega$ where this is wrong, for the exit time $\xi^n$ of $Z_t$ from $Q^n$ we have $\xi^n<T$ for all $n$. However owing to \eqref{eq4.2}, the sequence $X_{\xi^n}$ should be bounded, then the sequence $Z_{\xi^n}$ has limit points on the boundary $\partial Q$. According to the Assumption \ref{addvi} (vi), it only happens with probability zero. Hence, $a.s.$ there is $n\geq 1$ such that $T\leq\xi^n$. Since this happens for any $T$ and $\xi^n<\xi$ we conclude that $\xi=\infty$ $a.s.$, which proves our intermediate claim.\\
\indent Since $dist(\partial Q^m,\partial Q^{m+1})>0$ we can find $\kappa\in(0,1]$ sufficiently small so that $(s,x)+Q^{\kappa^2,\kappa}\subset Q^{m+1}$ for all $(s,x)\in Q^m$. Therefore, by translation and dilation, without losing generality, we may assume that $s=0$, $x=0$ and $Q^{1,1}\subset Q^m$.\\
\indent Next we notice that obviously, to prove \eqref{eq4.1} it suffices to prove that with $N$ of the same kind as in \eqref{eq4.1} for any $n\geq m+2$,
\begin{equation}\label{eq4.3}
E\sup_{t<\xi^n\wedge T}\exp(\mu \phi(t,X_t)+\mu \nu|X_t|^2)\leq N.
\end{equation}
\indent Fix an $n\geq m+2$. By virtue of Theorem \ref{mainadd1}, notice that the left-hand side of \eqref{eq4.3} will not change if we change  $-\sigma\sigma^*\nabla\phi+(\frac{1}{2}\sum_{j=1}^d\partial_j a_{ij})_{1\leq i\leq d}$ outside of $Q^n$.  Therefore we may replace $\phi$ with $\phi\eta$ and replace $\frac{1}{2}\sum_{j=1}^d\partial_j a_{ij}$ with $\frac{1}{2}\sum_{j=1}^d\partial_j a_{ij}\eta$  for each $1\leq i\leq d$, where $\eta$ is an infinitely differentiable function equal $1$ on a neighborhood of ${Q}^n$ and equals $0$ outside of $Q^{n+1}$. To simplify the notation we just assume that $\phi$ and $\frac{1}{2}\sum_{j=1}^d\partial_j a_{ij}$ vanishes outside of $Q^{n+1}$ and \eqref{eq5} holds in a neighborhood of $\overline{Q^n}$. This is harmless as long as we prove that $N$ depends appropriately on the data.\\
\indent Now we mollify $\phi$ by convolving it with a $\delta-$like nonnegative smooth function $\zeta^\gamma(t,x)=\gamma^{-d-1}\zeta(t/\gamma,x/\gamma)$, $\zeta$ has compact support in $Q^{1}$. Denote by $\phi^{(\gamma)}$ the result of the convolution and use an analogous notation for the convolution of $\zeta^\gamma(t,x)$ with other functions. Also denote by $(X_t^\gamma)_{t\geq0}$ the solution of the following SDE
 $$X_t^{\gamma}=\int_0^t\sigma(s,X_s^{\gamma})dW_s+\int_0^t(-\sigma\sigma^*\nabla\phi^{(\gamma)})(s,X_s^{\gamma})ds+
 (\frac{1}{2}\sum_{j=1}^d\int_0^{t}\partial_ja_{ij}(s,X_s^\gamma)ds)_{1\leq i\leq d}.$$
 For $x_{\cdot}\in \mathcal{C}([0,\infty),\mathbb{R}^{d})$ we define $\xi_n(x_\cdot):=\inf\left\{t\geq0:(t,x_t)\not\in Q^n\right\}$. Consider the bounded function $f$ on $\mathcal{C}([0,\infty),\mathbb{R}^{d})$ given by the formula
$$f(x_\cdot)=\sup_{t\leq\xi^n(x_\cdot)\wedge T}\exp(\mu\phi(t,x_t)+\mu \nu|x_t|^2),$$
 and let $f^\gamma$ be defined by the same formula with $\phi^{(\gamma)}$ in place of $\phi$.
 Since $\sigma\sigma^*$ is bounded, by using Lemma \ref{lemm3.6} we conclude that the left-hand side of \eqref{eq4.3} is equal to the limit as $\gamma\downarrow 0$ of
\begin{align}\label{approximation}
Ef^\gamma(X_\cdot^{\gamma})=E\sup_{t<\xi^n(X_\cdot^{\gamma})\wedge T}\exp(\mu\phi^{(\gamma)}(t,X_t^\gamma)+\mu \nu|X_t^\gamma|^2).
\end{align}
In fact, if we denote $M_t=\int_0^t\sigma(s,M_s)dW_s$, $t\geq0$, according to Lemma \ref{lemm3.6}
\begin{align*} |Ef(X_\cdot)-Ef^\gamma(X_\cdot^{\gamma})| &\leq N'(E|f(M_\cdot)-f^\gamma(M_\cdot)|^2)^{1/2}+N'\Vert f\Vert_{\infty}\Vert \sigma\sigma^*(\nabla\phi-\nabla\phi^{(\gamma)})I_{Q^n}\Vert_{L_p^q}\\ & \leq N'(E|f(M_\cdot)-f^\gamma(M_\cdot)|^2)^{1/2}+ KN'\Vert (\nabla\phi-\nabla\phi^{(\gamma)})I_{Q^n}\Vert_{L_p^q},  \end{align*}
which of course tends to $0$ when $\gamma\rightarrow 0$, since $\phi$ is continuous and bounded on $Q^n$, $|I_{Q^n}\nabla\phi|\in \mathbb{L}_p^q$, then $f^\gamma\rightarrow f$ and $I_{Q^n}\nabla\phi^{(\gamma)}\rightarrow I_{Q^n}\nabla\phi $ in $\mathbb{L}_p^q$ as $\gamma\rightarrow 0.$ \\
\indent In the light of the fact that \eqref{eq5} holds in a neighborhood of ${Q}^n$ we have that on $Q^n$ for sufficiently small $\gamma$
\begin{align}
2D_t\phi^{(\gamma)}+\sum_{i,j=1}^{d}\partial_j(a_{ij}\partial_i\phi^{(\gamma)})&\leq ((he^{\epsilon\phi})^{(\gamma)}e^{-\epsilon\phi^{(\gamma)}}+\sum_{i,j=1}^d|\partial_j(a_{ij}\partial_i\phi^{(\gamma)})-(\partial_j(a_{ij}\partial_i\phi))^{(\gamma)}|) e^{\epsilon\phi^{(\gamma)}}
\nonumber
\\&=:h^{\gamma}e^{\epsilon\phi^{(\gamma)}}.\label{eq4.4}
\end{align}
Since $h$ is continuous, then $(he^{\epsilon\phi})^{(\gamma)}e^{-\epsilon\phi^{(\gamma)}}\rightarrow h$ uniformly on $Q^n$. Besides $\sum_{i,j=1}^d|\partial_j(a_{ij}\partial_i\phi^{(\gamma)})-(\partial_j(a_{ij}\partial_i\phi))^{(\gamma)}|)\rightarrow 0$ pointwise. Hence if we
 denote
 $$H_{Q^n}^{\gamma}(T,(2-r\eta)\nu,r):=\int_{Q^n}(h^{\gamma})^r(t,x)I_{(0,T)}(t)e^{-(2-r\eta)\nu|x|^2}dtdx,$$
 we have
 \begin{align*}
 \lim_{\gamma\rightarrow 0}H_{Q^n}^{\gamma}(T,(2-r\eta)\nu,r)\leq  H_{Q^n}(T,(2-r\eta)\nu,r).
 \end{align*}
  Furthermore, the conditions $2D_t\phi^{(\gamma)}\leq K_1\phi^{(\gamma)}$
 %+\frac{1}{2}\sum_{i,j=1}^d(\partial_i\phi^{(\gamma)}+2\nu x_i)\partial_ja_{ij}
 %${\color{red}\frac{1}{2}\sum_{i,j=1}^d\partial_ja_{ij}(\partial_i\phi^{(\gamma)}}+{\color{red}2\nu x_i)\geq 0}$
 also hold in a neighborhood of ${Q}^n$ for sufficiently small $\gamma$.  \\
 %({\color{purple}It's true when $\partial_ia_{ij}=0$ (additive noise), but I don't know the general case, I also don't know how to make right assumption about $\phi$ and $\sigma$ in order to make it true when we only know $2D_t\phi+\frac{1}{2}\sum_{i,j=1}^d(\partial_i\phi+2\nu x_i)\partial_ja_{ij}\leq K_1\phi$ and $\frac{1}{2}\sum_{i,j=1}^d\partial_ja_{ij}(\partial_i\phi+2\nu x_i)\geq 0$})
\indent We now apply Theorem \ref{th3.12} for $Q^{n}\cap Q_{ T}$ in place of $Q$ to conclude that
\begin{align*} E\sup_{t<\xi^n\wedge T} \exp(\mu\phi(t,X_t)+\mu\nu|X_t|^2)&=\lim_{\gamma\downarrow0}E\sup_{t<\xi^n(X_\cdot^{\gamma})\wedge T}\exp(\mu\phi^{(\gamma)}+\mu \nu|X_t^{\gamma}|^2)\\&\leq \lim_{\gamma\downarrow0}(N+NH_{Q^n}^\gamma(T,(2-r\eta)\nu,r))\\&\leq N+NH_{Q^n}(T,(2-r\eta)\nu,r)\\&\leq N+NH_Q(T,(2-r\eta)\nu,r),\end{align*}
where the values of all the parameters are specified in \ref{th3.12} and the constants $N$ depend only on $r$, $d$, $p(m+1)$, $q(m+1)$, $\epsilon$, $T$, $K$, $K_1$, $\Vert I_{Q^{m+1}}\nabla\phi \Vert_{L_{p(m+1)}^{q(m+1)}}$, and $\sup_{Q^{m+1}}\left\{\phi+h\right\}.$\\
\indent We finally use condition $(H)$ from Assumption \ref{ass2.1}  . Fix any $r_0\in(1,2/(2\delta+\epsilon))$, set $a=(2-r_0\eta)\nu$ $(>0)$ and take $r=r(T,a)$ from condition $(H)$. H\"older's inequality shows that if condition $(H)$ is satisfied with $r=r'$ where $r'>1$, then it is also satisfied with any $r\in(1,r']$. Hence without losing generality we may assume that $r=r(T,a)\in(1,r_0]$. Then $(2-r\eta)\nu\geq a$ and $H_Q(T,(2-r\eta)\nu,r)\leq H_Q(T,a,r(T,a))<\infty.$ Thus, Theorem \ref{th3.12} yields \eqref{eq4.3}. The theorem is proved.\\
\indent $\hfill{} \Box$\\

\begin{remark}\label{remm4.1}\emph{
We can add another drift term to \eqref{eq2.4}, it does not have to be the gradient of a function. Under Assumption \ref{ass2.1}   take a Borel measurable locally bounded $\mathbb{R}^d$ valued function $b(t,x)$ defined on $\mathbb{R}^{d+1}$ satisfying the condition $|b(t,x)|\leq c(1+|x|)$, where $c$ is a finite positive constant, then it turns out that the first assertion of Theorem \ref{th2.3}  still holds with the equation
\begin{align}\label{addsde}
X_t=x+\int_0^t\sigma(s+r,X_r)dW_r&+\int_0^t(-\sigma\sigma^*\nabla\phi)(s+r,X_r)dr+\int_0^tb(s+r,X_r)dr\nonumber\\&+(\int_0^t\frac{1}{2}\sum_{j=1}^d\partial_ja_{ij}(s+r,X_r)dr)_{1\leq i\leq d},\quad t\geq0\quad
\end{align}
in place of \eqref{eq2.4}. To prove this we follow the proof in \cite{KR} Remark 8.2.
   The only needed material is the Markov property of solution to equation \eqref{eq2.4}, {which we already get from the proof of Theorem \ref{mainadd1}.} By applying Girsanov theorem we get the non-explosion result for the equation \eqref{addsde}.\\
  \indent Further we can carry our results in Theorem \ref{th2.3}  to the cases in which $\phi$ is not necessarily nonnegative but $\phi\geq -C(1+|x|^2)$, $C>0.$ Since the equation \eqref{eq2.4} is equivalent to the following
  \begin{align*}
  X_t=x&+\int_0^t\sigma(s+r,X_r)dW_r+(\frac{1}{2}\int_0^t\sum_{j=1}^d\partial_ja_{ij}(s+r,X_r)dr)_{1\leq i\leq d}\\&+\int_0^t2C\sigma\sigma^*(s+r,X_r)X_rdr
  -\int_0^t\sigma\sigma^*\nabla[C(1+|x|^2)+\phi](s+r,X_r)dr,\quad t\geq0,
  \end{align*}
  obviously $|\sigma\sigma^*(t,x)x|\leq K(1+|x|)$. We conclude that the SDE \eqref{eq2.4} has a unique solution defined for all times if $(s,x)\in Q$ provided that $\phi+C(1+|x|^2)$ rather than $\phi$ satisfies Assumption \ref{ass2.1}.}
  \end{remark}

\section{Examples and applications}
In this section, we will give several examples to show the local well-posedness and  non-explosion of solution to the SDE that our results can be applied. \\
\subsection{Examples-Maximal local well-posedness}

\begin{example}
Consider the equation \eqref{eq1} when $d=1$,  $Q=\mathbb{R}_{+}\times(0,\infty)$, $Q^n=(0,n)\times\left\{x:1/n<x<n\right\}$ for $n\in\mathbb{N}$, $b(t,x)=-x^{-1}$, $\sigma(t,x)=(1+x^2)^{-1}$. \\
\indent For any $(s,x)\in Q$, for any $n\in\mathbb{N}$, if we take $q(n)=\infty$ and $p(n)\in(2,\infty)$, then $1/p(n)+2/q(n)<1$. We can also easily check that $\Vert bI_{Q^n}\Vert_{\mathbb{L}_{p(n)}^\infty}<\infty$, and $\Vert I_{Q^n}\nabla \sigma \Vert_{\mathbb{L}_{p(n)}^\infty}<\infty$. Furthermore, $\sigma(t,x)$ is uniformly continuous in $x$ uniformly with respect to $t$ for $(t,x)\in Q^n$, and there exist positive constants $\delta_n(=(1+n^2)^{-2})$ such that for all $(t,x)\in Q^n$,
$$|\sigma^*(t,x)\lambda|^2\geq \delta_n|\lambda|^2, \quad \forall \lambda\in\mathbb{R}^d.$$
Hence by Theorem \ref{mainadd1} there exists an $(\mathcal{F}_t)$-stopping time $\xi$ and  a unique $(\mathcal{F}_t)$-adapted solution to the following equation
\begin{align*}
X_t=x-\int_0^t\frac{1}{X_r}dr+\int_0^t(1+X_r^2)^{-1}dW_r,\quad t\in[0,\xi).
\end{align*}
\end{example}
\begin{example}
If $d=2$ with $b(t,x)=x\ln|x^{(1)}|=(x^{(1)}\ln|x^{(1)}|,x^{(2)}\ln|x^{(1)}|)$, $\sigma(t,x)=I_{2}\cdot\ln (2+|x|^2)$ on $Q=\mathbb{R}_+\times\mathbb{R}^2\backslash\left\{x^{(1)}=0\right\}$ and $Q_n=(0,n)\times\left\{x\in\mathbb{R}^2:1/n<|x^{(1)}|< n, |x^{(2)}|<n\right\}$, where $x^{(i)}$ denotes the $i-$th exponent of the vector $x\in\mathbb{R}^d$ and $I_{2}$ is the identity matrix in $\mathbb{R}^{2}$. Then by Theorem \ref{mainadd1} for any $(s,x)\in Q$, there exist an $(\mathcal{F}_t)$-stopping time $\xi$ and  a unique $(\mathcal{F}_t)$-adapted solution to the following SDE
 \begin{equation*}
\left\{
\begin{aligned}
X_t^{(1)}=x^{(1)}+\int_0^t X_r^{(1)}\ln |X_r^{(1)}|dr+\int_0^t\ln (2+|X_r|^2)dW_r^{(1)},\\
X_t^{(2)}=x^{(2)}+\int_0^t X_r^{(2)}\ln |X_r^{(1)}|dr+\int_0^t\ln (2+|X_r|^2)dW_r^{(2)},
\end{aligned}
\right.
\end{equation*}
which can be rewritten as
\begin{align*}
X_t=x+\int_0^t X_r\ln |X_r^{(1)}|dr+\int_0^tI_2\ln (2+X_r^2)dW_r, \quad t\in[0,\xi).
\end{align*}
More precisely, for $n\in\mathbb{N}$, we can take $p(n)\in(2,\infty)$ and $q(n)=\infty$, then $\Vert bI_{Q^n}\Vert_{\mathbb{L}_{p(n)}^\infty}<\infty$, and $\Vert \partial \sigma I_{Q^n}\Vert_{\mathbb{L}_{p(n)}^\infty}<\infty$. Put $0<\delta_n<\ln^22$, then condition (ii) in Theorem \ref{mainadd1} also is fulfilled.
\end{example}
\subsection{Example-Non-explosion}
\begin{example}
Consider $d=1$, $Q=\mathbb{R}_{+}\times(0,\infty)$, and $Q^n=(0,n)\times\left\{x:1/n<x<n\right\}$, for $\delta>0$, let $\phi(t,x)=|x|^{-\delta}+|x|$,  $\sigma(t,x)=2+\sin x$.\\
\indent We can find that $\phi$ is a nonnegative continuous function on $Q$ and blows up near the parabolic boundary of $Q$. For $n\in\mathbb{N}$, take $q(n)=\infty$, $p(n)\in(2,\infty)$, then $1/p(n)+2/q(n)<1$ and $\Vert (-\sigma^2\nabla\phi+\sigma\nabla\sigma)I_{Q^n}\Vert_{\mathbb{L}_{p(n)}^\infty}<\infty$. Besides,
  $$\nabla(\sigma^2\nabla \phi)(t,x)\leq C e^{3/2\phi(t,x)}$$
  with constant $C\in(0,\infty)$. For $\sigma$, it can be easily checked that conditions in Assumption \ref{ass2.1} are satisfied. Then by Theorem \ref{th2.3} the following SDE has a unique $(\mathcal{F}_t)$-adapted solution on $Q$:
  \begin{align*}
  X_t=x+\int_0^t(2+\sin X_s)dW_s+\int_0^t (\delta X_s|X_s|^{-\delta-2}&-\frac{X_s}{|X_s|})(2+\sin X_s)^2ds\\&+\int_0^t(2+\sin X_s)\cos X_sds,\quad t\geq0.
  \end{align*}
\end{example}

\subsection{Diffusions in random media}
We apply our results to a particle which performs a random motion in $\mathbb{R}^d$, $d\geq2$, interacting with impurities which are randomly distributed according to a Gibbs measure of Ruelle type. So, the impurities form a locally finite subset $\gamma=\left\{x_k|k\in\mathbb{N}\right\}\subset\mathbb{R}^d$. The interaction is given by a pair potential $V$ and diffusion coefficient $\sigma$ to be specified below defined on $\left\{x\in\mathbb{R}^d:|x|>\rho\right\}$, where $\rho\geq0$ is a given constant. The stochastic dynamics of the particle is then determined by a stochastic equation type \eqref{eq2.4} as in Theorem \ref{th2.3} above with
\begin{align}\label{eq5.21}
Q:=\mathbb{R}_+\times(\mathbb{R}^d\backslash\gamma^\rho), \quad \phi(t,x):=\sum_{y\in\gamma}V(x-y),\quad (t,x)\in Q,
\end{align}
where $\gamma^\rho$ is the closed $\rho-$neighborhood of the set $\gamma$, i.e., the random path $(X_t)_{t\geq0}$ of the particle should be the strong solution of
\begin{align}\label{eq5.22}
X_t=x&+\int_0^t\sigma(X_s)dW_s
+(\frac{1}{2}\sum_{j=1}^d\int_0^t\partial_ja_{ij}(X_s)ds)_{1\leq i\leq d}
-\sum_{w\in\gamma}\int_0^t(\sigma\sigma^*)(X_s)\nabla V(X_s-w)ds.
\end{align}
Below we shall give conditions on the pair potential $V$ and diffusion coefficient $\sigma$ which imply that this is indeed the case, i.e. that Theorem \ref{th2.3} above applies, for all $\gamma$ outside a set of measure zero for the Gibbs measure. Here the original case is from \cite{KR} section 9.1, we generalize it to the multiplicative noise case. Similarly the set of admissible impurities $\gamma$ we can treat is
\begin{align}\label{eq5.23}
\Gamma_{ad}:=\left\{\gamma\subset\mathbb{R}^d|\forall r>0 \exists c(\gamma,r)>0:|\gamma\cap B_r(x)|\leq c(\gamma,r)\log(1+|x|), \forall x\in\mathbb{R}^d\right\},
\end{align}
where $B_r(x)$ denotes the open ball with center $x$ and radius $r$, $|A|$ denotes the cardinality of a set $A$. From \cite{KR} we know that for essentially all classes of Gibbs measure in equilibrium statistical mechanics of interacting infinite particle systems in $\mathbb{R}^d$ the set $\Gamma_{ad}$ has measure one, this is also true for Ruelle measures.\\
\indent We fix a $\gamma\in\Gamma_{ad}$. The necessary conditions on the pair potential $V$ and diffusion coefficient $\sigma$ go as follows (the typical case when $\rho=0$ is also included):\\
\noindent(V1) The function $V$ is positive and once continuously differentiable in $\mathbb{R}^d\cap\left\{|x|>\rho\right\}$, $\lim_{|x|\downarrow \rho}V(x)=\infty$.\\
\noindent(V2) There exist finite constants $\alpha>d/2$, $C\geq0$, $\epsilon\in[1,2)$ such that with $U(x)=:C(1+|x|^2)^{-\alpha}$ we have
\begin{align}\label{eq5.24}
|V(x)|+|\nabla V(x)|\leq U(x) \quad \text{for} \quad |x|> \rho,
\end{align}
and for any $|y|>\rho$
\begin{align}\label{eq5.25}
\sum_{i,j=1}^d(\partial_ja_{ij}(x)\partial_iV(y)+a_{ij}(x)\partial_{i}\partial_jV(y))\leq C(e^{\epsilon(V+U)(y)}-1)
\end{align}
in the sense of distributions on $\left\{x\in\mathbb{R}^d:|x|>\rho\right\}$ where $\sigma(x)=(\sigma_{ij}(x))_{1\leq i,j\leq d}:\mathbb{R}^d\rightarrow\mathbb{R}^d\times\mathbb{R}^d$ satisfies the following conditions:\\
\noindent ($\sigma$1) There exists a positive constant $K$ such that for all $x\in\mathbb{R}^d$
 \begin{align}\label{eq5.25}
 \frac{1}{K}|\lambda|^2\leq\langle(\sigma\sigma^*)(x)\lambda,\lambda\rangle\leq K|\lambda|^2,\quad \forall \lambda\in\mathbb{R}^d.
 \end{align}
\noindent ($\sigma$2) For $1\leq i$, $j\leq d$, $\sigma_{ij}\in\mathcal{C}_b^2(\mathbb{R}^d)$.\\

\indent We emphasize that above conditions are fulfilled for essentially all potentials of interests in statistical physics. \\
\indent Introduce $\bar V(x)=V(x)+2U(x)$, $|x|>\rho$, and for $(t,x)\in Q$ let
\begin{align*}
\bar \phi(t,x):=\sum_{y\in\gamma}\bar V(x-y),\quad ( a_{ij})_{1\leq i,j\leq d}:=\sigma\sigma^*, \\b(t,x):=2\sum_{w\in\gamma}(\sigma\sigma^*)(x)\nabla U(x-w).\quad\quad
\end{align*}
Owing to \eqref{eq5.24}, \eqref{eq5.25} and the fact that $\gamma\in\Gamma_{ad}$, the function $\phi$ is continuously differentiable in $Q$ and $|b(t,x)|\leq NK\log(2+|x|)$, where $N$ is independent of $(t,x)$ (See \cite{KR} Section 9.1). Meanwhile for appropriate constants $N$ on $Q$ we have for $|y|>\rho$

$$2\sum_{i,j=1}^d(\partial_ja_{ij}(x)\partial_iU(y)+a_{ij}(x)\partial_j\partial_iU(y))\leq N(e^{\epsilon U(y)}-1) $$
 because of conditions $(\sigma1)$ and $(\sigma2)$. Combing this with the fact that  $V+U$ is positive and $\sum(e^{a_k}-1)\leq e^{\sum a_k}-1$, $a_k\geq 0$, we find that there exists a constant $N'>0$ independent of $(t,x)$ such that
\begin{align*}
\sum_{i,j=1}^d\partial_j( a_{ij}\partial_i \bar \phi)(x)&=\sum_{i,j=1}^d\sum_{w\in\gamma}\partial_j(a_{ij}(x)\partial_i (V(x-w)+2U(x-w)))
\\&\leq N\sum_{w\in\gamma}\Big((e^{\epsilon(V(x-w)+2U(x-w))}-1)+(e^{\epsilon U(x-w)}-1)\Big)\leq N'(e^{\epsilon \bar \phi(x)}-1).
\end{align*}
It shows that all conditions on $\bar\phi$ and $\sigma$ in Theorem \ref{th2.3} are fulfilled and therefore by Remark \ref{remm4.1} the equation
\begin{align}\label{bar}
X_t=x+\int_0^t\sigma(X_s)dW_s-\int_0^t(\sigma\sigma^*\nabla\bar \phi)(X_s)ds+(\frac{1}{2}\sum_{j=1}^d\int_0^t\partial_j a_{ij}(X_s)ds)_{1\leq i\leq d}+\int_0^tb(X_s)ds
\end{align}
has a unique strong solution defined for all times if $x\in\mathbb{R}^d\backslash\gamma^\rho$. Since equation \eqref{bar} coincides with SDE \eqref{eq5.22}, we get the desired conclusion.

  \subsection{M-particle systems with gradient dynamics}
  In this subsection we consider a model of $M$ particles in $\mathbb{R}^d$ interacting via a pair potential $V$ and diffusion coefficient $\sigma$ satisfying the following conditions:\\
  (V1) The function $V$ is once continuously differentiable in $\mathbb{R}^d\backslash\left\{0\right\}$, $\lim_{|x|\rightarrow0}V(x)=\infty$, and on $\mathbb{R}^d\backslash\left\{0\right\}$ we assume that $V\geq -U$, where $U(x):=C(1+|x|^2)$, $C$ is a constant.\\
  (V2) There exists a constant $\epsilon\in[1,2)$ such that for arbitrary $x,y\in\mathbb{R}^d\backslash\left\{0\right\}$,
  \begin{align}\label{eq5.1}
  \sum_{i,j=1}^d(\partial_ja_{i,j}(x)\partial_iV(y)+a_{i,j}(x)\partial_i\partial_jV(y))
 \leq Ce^{\epsilon(V+U)(y)}
  \end{align}
  in the sense of distributions. \\
  Here $(a_{i,j})_{1\leq i,j\leq d}:=\sigma\sigma^*$ and $\sigma(x)=(\sigma_{i,j}(x))_{1\leq i,j\leq d}:\mathbb{R}^d\rightarrow\mathbb{R}^d\times\mathbb{R}^d$ is the diffusion coefficient satisfying:\\
 ($\sigma$1) There exists a positive constant $K$ such that for all $x\in\mathbb{R}^d$
 \begin{align*}
 \frac{1}{K}|\lambda|^2\leq\langle(\sigma\sigma^*)(x)\lambda,\lambda\rangle\leq K|\lambda|^2,\quad \forall \lambda\in\mathbb{R}^d,
 \end{align*}
($\sigma$2) For $1\leq i$, $j\leq d$, $\sigma_{i,j}\in\mathcal{C}_b^2(\mathbb{R}^d)$.\\
  \indent Introduce $\bar V:=V+2U$,
  $$Q:=\mathbb{R}_+\times\left(\mathbb{R}^{Md}\backslash \cup_{1\leq k<j\leq M}\left\{x=(x^{(1)},...,x^{(M)})\in\mathbb{R}^{Md}:x^{(k)}=x^{(j)}\right\}\right),$$
  $$Q^n:=(0,n)\times \left\{x=(x^{(1)},...,x^{(M)})\in\mathbb{R}^{Md}:|x|<n,x^{(k)}\neq x^{(j)}\text{ for }1\leq k<j\leq M\right\},$$
 and let the function $\phi$, $\bar\phi$, $\bar \sigma$, $\bar a$ and $b$ be defined on $Q$ by
 $$
 \phi(t,x):=\sum_{1\leq k<j\leq M}V(x^{(k)}-x^{(j)}),\quad \bar\phi(t,x):=\sum_{1\leq k<j\leq M}\bar V(x^{(k)}-x^{(j)}),
  $$
\begin{align*}\bar\sigma(t,x):=
\left[\begin{matrix}
\sigma(x^{(1)})&0&0\\
0&\sigma(x^{(2)})&0\\
\cdots&\cdots&\cdots\\
0&0&\sigma(x^{(M)})
\end{matrix}\right],\quad \bar a(t,x):=\left[\begin{matrix}
(\sigma\sigma^*)(x^{(1)})&0&0\\
0&(\sigma\sigma^*)(x^{(2)})&0\\
\cdots&\cdots&\cdots\\
0&0&(\sigma\sigma^*)(x^{(M)})
\end{matrix}\right],\end{align*}
  %$$(\bar \sigma_{i,hk}(t,x))_{1\leq i\leq M,1\leq h,k\leq d}:=(\bar \sigma_{h+(i-1)d,k+(i-1)d}(t,x))_{1\leq i\leq M,1\leq h,k\leq d}:=(\sigma_{h,k}(x^{(i)}))_{1\leq i\leq M,1\leq h,k\leq d},$$
 $$
 \quad b:=(b^{(1)},...,b^{(M)}),\quad b^{(k)}(t,x):=4C(\sigma\sigma^*)(x^{(k)})\sum_{1\leq j\neq k\leq M}(x^{(k)}-x^{(j)}),\quad k=1,\cdots,M.
$$ Observe that for arbitrary $y$, $x\in\mathbb{R}^d\backslash\left\{0\right\}$,
    $$2\sum_{i,j=1}^d(\partial_ja_{i,j}(x)\partial_iU(y)+a_{i,j}(x)\partial_j\partial_iU(y))\leq Ne^{\epsilon U(y)}$$
    for an appropriate constant $N$ which is independent of $y$, $x$. Besides $\phi$ and $\bar\phi$ are continuously differentiable on $Q$.
   If we use the notation $\partial_{r}^kf(x):=\partial_{r}^kf((x^{(1)},\cdots,x^{(M)})):=\frac{\partial f((x^{(1)},\cdots,x^{(M)})) }{\partial x_r^{(k)}}$  for $k=1,\cdots, M$ and $r=1,\cdots,d$, then for $x\in\mathbb{R}^{Md}$,
  \begin{align}
  \bar a_{i,j}(t,x)&=\sum_{k=1}^Ma_{i-(k-1)d,j-(k-1)d}(x^{(k)})I_{(k-1)d< i,j\leq kd},\label{a}\\
 %\partial_{j}^k\bar a_{ij}(t,x)&= \partial_{rj}\bar a_{i+(k-1)d,j+(k-1)d}=\partial_ja_{i,j}(x^{(k)}),
  \partial_{r}^k\bar a_{i,j}(t,x)&=\partial_{r}^ka_{i-(k-1)d,j-(k-1)d}(x^{(k)})I_{(k-1)d< i,j\leq kd}=\partial_{r}a_{i-(k-1)d,j-(k-1)d}(x^{(k)})I_{(k-1)d< i,j\leq kd},\label{partiala}
  %I_{i\geq d\lfloor \frac{i+j}{d}\rfloor,j\geq d\lfloor \frac{i+j}{d}\rfloor}\sum_{l=1}^{Md}\Big(&\partial_r\sigma_{i-kd,l-kd}(x^{(k)})\sigma_{j-d \lfloor\frac{j}{d}\rfloor,l-d \lfloor\frac{l}{d}\rfloor}(x^{(\lfloor\frac{l+j}{d}\rfloor)})\nonumber\\&+\partial_r\sigma_{j-kd,l-kd}(x^{(k)})\sigma_{i-d \lfloor\frac{i}{d}\rfloor,l-d \lfloor\frac{l}{d}\rfloor}(x^{(\lfloor\frac{l+i}{d}\rfloor)})\Big)
  \end{align}
   where $1\leq i,j\leq Md$, and
   \begin{align*}
 \partial_{r}^k\bar\phi(t,x)=\sum_{1\leq q \neq k\leq M}\partial_rV((x^{(k)}-x^{(q)})sign(q-k))sign(q-k)+4C\sum_{1\leq q\neq k\leq M}(x_r^{(k)}-x_r^{(q)}),
   \end{align*}
   furthermore,
   \begin{align*}
  \partial_{n}^m\partial_{r}^k\bar\phi(t,x)=\sum_{1\leq q\neq k \leq M}\Big(&I_{m=k}\partial_n\partial_rV((x^{(k)}-x^{(q)})sign(q-k))\\&-I_{m=q}\partial_n\partial_rV((x^{(k)}-x^{(q)})sign(q-k))\Big)+4C(I_{m=k,n=r}-I_{m\neq k,n=r}).
   \end{align*}
   Combining the above equalities with our assumptions of $V$ and $\sigma$, by algebraic calculation we get that on $Q$ there exists a large number $C_{M,d}$ depending on $Md$ and a constant $C'\in(0,\infty)$ such that
   \begin{align*}
   2D_t&\bar\phi(t,x)+\sum_{i,j=1}^{Md}\partial_{j}(\bar a_{i,j}\partial_{i}\bar\phi)(t,x)
   \\=&\sum_{i,j=1}^d\sum_{k=1}^M\Big(\partial_j^ka_{i,j}(x^{(k)})\partial_{i}^{k}\bar\phi(t,x)+a_{i,j}(x^{(k)})\partial_j^k\partial_{i}^{k}\bar\phi(t,x)\Big)
   %&\sum_{n,r=1}^d\sum_{k=1}^m\sum_{m=1}^m\Big(\partial_{r}^k a_{i,r}(x^{(k)})\partial_{n}^m\bar\phi(t,x)+ \bar %a_{i,j}(t,x)\partial_{n}^m\partial_{r}^k\bar\phi(t,x)\Big)
   \\=&\sum_{i,j=1}^d\sum_{k=1}^M\sum_{1\leq q\neq k\leq M}\Big(\partial_ja_{i,j}(x^{(k)})[\partial_iV((x^{(k)}-x^{(q)})sign(q-k))sign(q-k)+4C(x_i^{(k)}-x_i^{(q)})]\\
   &\quad\quad\quad\quad\quad+a_{i,j}(x^{(k)})[\partial_j\partial_iV((x^{(k)}-x^{(q)})sign(q-k))]\Big)+\sum_{i,j=1}^d\sum_{k=1}^Ma_{i,j}(x^{(k)})4CI_{i=j}
   \\ \leq &C_{M,d}\sum_{1\leq q<g\leq M}(Ce^{\epsilon(V(x^{(q)}-x^{(g)})+U(x^{(q)}-x^{(g)}))}+Ne^{\epsilon(U(x^{(q)}-x^{(g)}))})
   \leq  C'e^{\epsilon\bar \phi(t,x)}.
   \end{align*}
   The continuity of $\bar a_{i,j}(t,x)$ on $Q$ and $\partial_{j}^k\bar a_{i,j}(t,x)$ on $Q^n$ can be easily checked from \eqref{a} and \eqref{partiala} and conditions about $\sigma$. In order to reduce the lengthy algebraic computation, we only show the part for  $\bar a_{i,j}(t,x)$, similarly we can get the desired continuity for $\partial_{j}^k\bar a_{i,j}(t,x)$ on $Q^n$.  For any $(t,x)$ and $(s,y)\in Q$, by \eqref{a} we have for $1\leq i,j\leq Md,$
   \begin{align*}
   |\bar a_{i,j}(t,x)-\bar a_{i,j}(s,y)|&\leq C_{Md}\sum_{k=1}^M|a_{i-(k-1)d,j-(k-1)d}(x^{(k)})-a_{i-(k-1)d,j-(k-1)d}(y^{(k)})|I_{(k-1)d< i,j\leq kd}
   %\sum_{1\leq k<h\leq M}\sum_{1\leq m<n\leq M}\sigma_{il}(x^{(k)}-x^{(h)})\sigma_{jl}(x^{(m)}-x^{(n)})\\&\quad\quad\quad-\sum_{1\leq k<h\leq M}\sum_{1\leq m<n\leq M}\sigma_{il}(y^{(k)}-y^{(h)})\sigma_{jl}(y^{(m)}-y^{(n)})\Big|
   %\\ &\leq \sum_{l=1}^d\sum_{1\leq k<h\leq M}\sum_{1\leq m<n\leq M}\Big(|(\sigma_{il}(x^{(k)}-x^{(h)})-\sigma_{il}(y^{(k)}-y^{(h)}))\sigma_{jl}(x^{(m)}-x^{(n)})|
   %\\&\quad\quad\quad+|\sigma_{il}(y^{(k)}-y^{(h)})(\sigma_{jl}(x^{(m)}-x^{(n)})-\sigma_{il}(y^{(m)}-y^{(n)}))|\Big)
   %\\ &\leq \sum_{l=1}^d\sum_{1\leq k<h\leq M}\sum_{1\leq m<n\leq M} K\Big(|(\sigma_{il}(x^{(k)}-x^{(h)})-\sigma_{il}(y^{(k)}-y^{(h)}))|
   %\\&\quad\quad\quad+|(\sigma_{jl}(x^{(m)}-x^{(n)})-\sigma_{il}(y^{(m)}-y^{(n)}))|\Big)
   \\&\leq C_{Md}\sum_{k=1}^M|x^{(k)}-y^{(k)}|\leq C'' |x-y|.
   \end{align*}
   We can adjust constants $C''$ and $K$ such that there is still a positive constant such condition $(\sigma1)$ satisfied.\\
 \indent It follows that all conditions on $\bar\phi$ and $\bar\sigma$ in Theorem \ref{th2.3} are fulfilled and therefore by Remark \ref{remm4.1} the corresponding stochastic equation for a process $(X_t)_{t\geq0}=(X_t^{(1)},...,X_t^{(M)})_{t\geq0}$ has a unique strong solution defined for all times whenever for the initial condition $x$ we have $(0,x)\in Q$. The corresponding equation is the following system
 \begin{align*}
  X_t^{(k)}=x^{(k)}+\int_0^t\sigma(X_s^{(k)}) dW_s^{(k)}-&\int_0^t(\sigma\sigma^*)(X_s^{(k)})\partial_k\bar\phi(s,X_s)ds\\
  &+(\frac{1}{2}\sum_{j=1}^d\int_0^t\partial_{j} a_{i,j}(X_s^{(k)})ds)_{1\leq i\leq d}+\int_0^tb^{(k)}(s,X_s)ds.
 \end{align*}
 We rewrite it as following with $k=1,...,M$
 \begin{align*}
 X_t^{(k)}=x^{(k)}+&\int_0^t\sigma(X_s^{(k)})dW_s^{(k)}\\-&\int_0^t(\sigma\sigma^*)(X_s^{(k)})\sum_{j=1,j\neq k}^M \nabla V((X_s^{(k)}-X_s^{(j)})sign(j-k))sign(j-k)ds
 \\+&(\frac{1}{2}\sum_{j=1}^d\int_0^t\partial_{j} a_{i,j}(X_s^{(k)})ds)_{1\leq j\leq d},
 \end{align*}
  which has a unique strong solution defined for all times whenever $(0,(x^{(1)},...,x^{(M)}))\in Q$.

  \iffalse If $V$ and $\sigma$ is symmetric: $V(x)=V(-x)$, the above equation becomes
  \begin{align*}
   X_t^{(k)}=x^{(k)}+&\int_0^t\sum_{1\leq m<n\leq M}\sigma(X_s^{(m)}-X_s^{(n)}) dW_s^{(k)}\\-&\int_0^t(\sum_{1\leq i<j\leq M}\sigma(X_s^{(i)}-X_s^{(j)}))^2\sum_{j=1,j\neq k}^M\partial V(X_s^{(k)}-X_s^{(j)})ds
   \\+&\frac{1}{2}\sum_{j=1}^d\sum_{l=1}^d\Big[(\sum_{1\leq m\neq n\leq M}\partial_j\sigma_{il}\Big(X_s^{(m)}-X_s^{(n)}sign(m-n)\Big)sign(m-n))\\&\quad\quad\quad\quad\quad\quad
 \quad\quad\quad\quad\quad\quad\quad\quad\quad\quad\quad\quad\quad\quad\quad\quad\quad\quad(\sum_{1\leq m<n\leq M}\sigma_{jl}(X_s^{(m)}-X_s^{(n)}))\\&\quad\quad\quad+(\sum_{1\leq m\neq n\leq M}\partial_j\sigma_{jl}\Big(X_s^{(m)}-X_s^{(n)})sign(m-n)\Big)sign(m-n))
 \\&\quad\quad\quad\quad\quad\quad
 \quad\quad\quad\quad\quad\quad\quad\quad\quad\quad\quad\quad\quad\quad\quad\quad\quad\quad(\sum_{1\leq m<n\leq M}\sigma_{il}(X_s^{(m)}-X_s^{(n)}))\Big].
  \end{align*}
   \fi

\renewcommand\thesection{A}
\section{Appendix}
\iffalse
\begin{lemma}\label{non-ex}
$X_t$  is a processes in $(\Omega,\mathcal{F},(\mathcal{F}_t)_{t\geq0},P)$, $\tau$ $=:\inf\left\{t\geq0:|X_t|=\infty\right\}$. Process $X_t$ is non-explosive $(i.e. $ $\tau=\infty$ $ a.s.)$ if for any $t>0$ one of the following conditions holds:\\
(i) $E|X_{t\wedge\tau}|\leq C(t)$.\\
(ii) $E|X_{t}|\leq C(t).$\\
\end{lemma}
\begin{proof} $\Omega=:\left\{\omega:\tau(\omega)=\infty\right\}$, $\Omega_n=:\left\{\omega:\tau(\omega)>n\right\}$. Then $\Omega=\cap_{n=1}^\infty\Omega_n$, and  $\Omega_n=\left\{\omega:|X_{n\wedge\tau(\omega)}(\omega)|<\infty\right\}=\left\{\omega:|X_{n}(\omega)|<\infty\right\}$. Since $T$ is arbitrary in $(0,\infty)$, for any $N\in \mathbb{N}$, $n\in [0,N]$, from condition (i) we get $E|X_{n\wedge\tau}|\leq C(N)$, it implies $|X_{n\wedge \tau}|<\infty$ $a.s.$, i.e. $P(\Omega_n)=1$.
Hence $P(\Omega)=1$, which implies $\tau=\infty$ $a.s.$. (i) is proved.\\
\indent For the second one, from condition (ii) we get $E|X_{n}|\leq C(N)$,  so $|X_{n}|<\infty$ $a.s.$,   then $P(\Omega_n)=1$, $P(\Omega)=1$, $\tau=\infty$ $a.s.$.\\
\end{proof}\\
\indent $\hfill{} \Box$\\
\fi

\begin{lemma}\label{lemm6.1}(\cite[P. 1 Lemma 1.1.]{NI})
Let $\left\{\beta(t)\right\}_{t\in[0,T]}$ be a nonnegative measurable $(\mathcal{F}_t)_{t\geq0}-$adapted process. Assume that for all $0\leq s\leq t\leq T$,
$$E\bigg(\int_s^t\beta(r)dr\bigg|{\mathcal{F}_s}\bigg)\leq\Gamma(s,t),$$
where $\Gamma(s,t)$ is a nonrandom interval function satisfying the following conditions:\\
(i) $\Gamma(t_1,t_2)\leq\Gamma(t_3,t_4)$ if $(t_1,t_2)\subset(t_3,t_4);$\\
(ii) $\lim_{h\downarrow 0}\sup_{0\leq s<t\leq T,|t-s|\leq h} \Gamma(s,t)=\lambda$, $\lambda\geq0$.
Then for any real $\kappa<\lambda^{-1}$ $($ if $\lambda=0$, then $\lambda^{-1}=\infty)$,
$$Eexp\left\{\kappa\int_0^T\beta(r)dr\right\}\leq C=C(\kappa, \Gamma, T)<\infty.$$\\
\end{lemma}

For the convenience of the reader, we include the $\mathcal{C}^\infty$-Urysohn Lemma  here.
\begin{lemma}\label{Urysohn} (\cite[8.18]{Folland} )
If $K\subset\mathbb{R}^n$ is compact and $U$ is an open set containing $K$, there exists smooth function $f$ such that $0\leq f\leq 1$, $f=1$ on $K$, and $supp(f)\subset U$.
\end{lemma}

%\indent $\hfill{} \Box$\\
The following lemma is based on a consequence of 7.6.4 in \cite{Robert}. We use this result a couple of times and hence for the sake of completeness we state it here precisely.
\begin{lemma}\label{Girsanov}
Let $\sigma$ and $b^{(i)}$, $i=1,2$ satisfy the conditions in Lemma \ref{lemm3.1}.  Let $(X_t^{(i)}, W_t^{(i)})_{t\geq 0}$ satisfy:
 $$X_t^{(i)}=x+\int_0^tb^{(i)}(s,X_s^{(i)})ds+\int_0^t\sigma(s,X_s^{(i)})dW_s^{(i)}.$$
 Then for any bounded Borel functions $f$ given on $\mathcal{C}=:\mathcal{C}([0,\infty),\mathbb{R}^d)$ we have
  $$Ef(X_\cdot^{(2)})=Ef(X_\cdot^{(1)})\overline{\rho}_\infty$$
  if
  \begin{align}\label{Novikov}
  E\exp\Big(\frac{1}{2}\int_0^\infty(\Delta b^*(s,X_s^{(1)})(\sigma\sigma^*)^{-1}(s,X_s^{(1)})\Delta b(s,Xs^{(1)}))ds\Big)<\infty,
  \end{align}
 where $\Delta b(t,X_t^{(1)}):=b^{(2)}(t,X_t^{(1)})-b^{(1)}(t,X_t^{(1)})$ and
 \begin{align*}\overline{\rho}_t:=\exp(&\int_0^t\Delta b^*(s,X_s^{(1)})(\sigma^*)^{-1}(s,X_s^{(1)})dW_s^{(1)}\\&-\frac{1}{2}\int_0^t(\Delta b^*(s,X_s^{(1)})(\sigma\sigma^*)^{-1}(s,X_s^{(1)})\Delta b)(s,X_s^{(1)})ds),\quad t\geq 0.
 \end{align*}
\end{lemma}
 \begin{proof}
  Theorem 6.1 in \cite{Robert} says if \eqref{Novikov} (Novikov condition) holds, then $(\overline{\rho}_t)_{t\geq0}$ is an $(\mathcal{F}_t)-$ martingale. Let $\hat P=\overline{\rho}_\infty P$, then $\hat P$ is also a probability on $(\Omega,\mathcal{F})$. By Theorem 4.1 in \cite{Ikeda},
  \begin{align*}
  \hat W_t=W^{(1)}_t-\int_0^t\sigma^{-1}(s,X_s^{(1)})\Delta b(s,X_s^{(1)})ds,\quad t\geq0
  \end{align*}
  is a $(\mathcal{F}_t)$-Brownian motion on the probability space $(\Omega,\mathcal{F},\hat P)$. So we can wirte
  \begin{align*}
  X_t^{(1)}&=x+\int_0^tb^{(1)}(s,X_s^{(1)})ds+\int_0^t\sigma(s,X_s^{(1)})d\hat W_s+\int_0^t\sigma(s,X_s^{(1)})\sigma^{-1}(s,X_s^{(1)})\Delta b(s,X_s^{(1)})ds
  \\&=x+\int_0^tb^{(1)}(s,X_s^{(1)})ds+\int_0^t\sigma(s,X_s^{(1)})d\hat W_s+\int_0^t\Delta b(s,X_s^{(1)})ds
  \\&=x+\int_0^tb^{(2)}(s,X_s^{(1)})ds+\int_0^t\sigma(s,X_s^{(1)})d\hat W_s,\quad t\geq 0.
  \end{align*}
  This implies that $(X^{(1)}_t,\hat W_t)_{t\geq0}$ is a  solution to the SDE
  \begin{align}\label{sde2}
  X_t^{(2)}=x+\int_0^tb^{(2)}(s,X_s^{(2)})ds+\int_0^t\sigma(s,X_s^{(2)})dW_s^{(2)},\quad t\geq 0,
  \end{align} on the probability space $(\Omega,\mathcal{F}, (\mathcal{F}_t)_{t\geq 0}, \hat P)$.
  From Lemma \ref{lemm3.1} we know that the solution to SDE \eqref{sde2} is unique, hence for any bounded Borel functions $f(x)$, given on $\mathcal{C}=:\mathcal{C}([0,\infty),\mathbb{R}^d)$ we have
  \begin{align*}
  Ef(X^{(2)}_\cdot)=\hat Ef(X^{(1)}_\cdot)=E\overline{\rho}_\infty f(X^{(1)}_\cdot).
  \end{align*}

 \end{proof}


\begin{thebibliography}{1000}

\bibitem{LB} L. Breiman: Probability. Society for Industrial and Applied Mathematics. Classics in Applied Mathematics (1992)

%\bibitem{PB} P. Billingsley: Probability and Measure. Anniversary Edition, WILEY. A John Wiley \& Sons, Inc., WILEY SERIES IN PROBABILITY AND STATISTIC. 2012.
%\bibitem{Da Prato} G. ~Da Prato: Introduction to Stochastic Analysis and Malliavin Calaulus. 2007
\bibitem{ES} H. Engelbert and W. Schmidt: Strong Markov continuous local martingales and solutions of one-dimensional stochastic differential equations, I, II, III. Math. Nachr., 143 (1989, 1991), 167-184; 144, 241-281; 151, 149-197.
\bibitem{Folland} G. B. Folland: Real Analysis:  Modern Techniques and Their Applications. A Wiley-Interscience Publication JOHN WILEY and SONS, INC.  1999
\bibitem{FF} E. Fedrizzi  and F. Flandoli : Pathwise uniqueness and continuous dependence of SDEs with non-regular drift. Stochastics, {83}(3) (2011), 241-257.
\bibitem{I.N.} I.~Gyongy and N.V.~Krylov: On the rate of convergence of splitting-up approximations for SPDEs. Progress in Probability. 56(2003), 301–321. Birkhauser Verlag, Basel

\bibitem{Ikeda} N.~Ikeda and S.~Watanabe: Stochastic Differential Equations and Diffusion Processes. North-Holland Publishing Company, Amsterdam Oxford New York (1981)

\bibitem{KR} N.V.~Krylov and M.~R\"ockner: Strong solutions of stochastic equations with singular time dependent drift. Probab. Theory Related Fields 131 (2005), no.2, 154-196. MR
\bibitem{pdeb} O.A. Lady$\breve{z}$enskaja, V.A. Solonnikov and N.N. Ural\'ceva: Linear and Quasi-linear Equations of Parabolic Type. American Mathematical Society. 1968
 \bibitem{Trutnau} H.~Lee and G.~Trutnau: Existence and uniquness results for It\^o's-SDEs with locally integrable drifts and Sobolev diffusion coefficients. https://arxiv.org/abs/1708.01152

 \bibitem{LWMR} W. Liu and M. R\"ockner: Stochastic Partial Differential Equations: An Introduction. Springer. Universitext.

 \bibitem{Robert} R.S.~Liptser and A.N.~Shiryaev: Statistics of Random Processes. "Nauka", Moscow, 1974 in Russian, English translation: Springer-Verlag, Berlin and New York, 1977
\bibitem{BO} B. $\varnothing$ksendal. Stochastic Differential Equations. An Introduction with Applications. Springer-Verlag Heidelberg New York (2000).
 \bibitem{NI} N.I. Portenko: Generalized diffusion processes.   "Nauka", Moscow, 1982 in Russian, English translation: Amer. Math. Soc. Provdence, Rhode Island, 1990

\bibitem{LOTFI RIAHI}L.~Riahi: Comparison of Green functions and harmonic measures for parabolic operators.  Potential Anal. 23 (2005), no. 4, 381–402. MR
%\bibitem{DW} D.W.~Stroock and S.R.Srinivasa Varadhan: Multidimensional Diffusion Processes. Springer Berlin Heidelberg New York (2000).
\bibitem{A. J.}  A.J.~Veretennikov: On the strong solutions of stochastic differential equations. Theory Probab. Appl. 24 (1979), 354-366

\bibitem{A. Yu}A.Yu.~Veretennikov: On strong solution and explicit formulas for solutions of stochastic integral equations. Math. Ussr Sb. 39 (1981), 387-403
\bibitem{K} K. Von der L\"uhe: Pathwise uniqueness for stochastic differential equations with singular drift and nonconstant diffusion. Phd thesis. Bielefeld University. 2018.
\bibitem{Wang}F. Wang: Integrability conditions for SDEs and semi-linear SPDEs. Ann. Probab. 45 (2017), no.5, 3223-3265. MR.
\bibitem{XXZZ}P. Xia, L. Xie, X. Zhang  and  G. Zhao: $L^q(L^p)$-theory of stochastic differential equations. https://arxiv.org/pdf/1908.01255.pdf
\bibitem{Zhang Xie}L. Xie and X. Zhang: Ergodicity of stochastic differential equations with jumps and singular coefficients. https://arxiv.org/pdf/1705.07402.pdf.
    \bibitem{YW} T. Yamada and S. Watanabe: On the uniqueness of solutions of stochastic differential equations. I, II. J. Math. Kyoto Univ. 11 (1971), 155-167, 553-563
        \bibitem{Zhang 2005} X. Zhang: Strong solutions of SDEs with singular drift and Sobolev diffusion coefficients. Stochastic Process. Appl. 115 (2005), no. 11, 1805-1818. MR
    \bibitem{Zhang2011} X. Zhang: Stochastic homeomorphism flows of SDE with singular drifts and Sobolev diffusion coefficients.  Electron. J. Probab. 16 (2011), no.38, 1096-1116. MR
         \bibitem{Zvonkin}A.K. Zvonkin: A transformation of the phase space of a diffusion process that removes the drift.  Mat. Sbornik,
93 (135) (1974), 129-149.
\end{thebibliography}
\end{document}